%% file: FreydCats.tex
\documentclass[12pt]{amsart}

\usepackage[pdfauthor   = {Sebastian\ Posur},
            pdftitle    = {A\ constructive\ approach\ to\ Freyd\ categories},
            pdfsubject  = {},
            pdfkeywords = {Freyd category;\ finitely presented functor;\ computable abelian category},
            bookmarks=true,
            bookmarksopen=true,
            pagebackref=true,
            hyperindex=true,
            colorlinks=true,
            linkcolor=blue,
            citecolor=blue,
            filecolor=blue,
            urlcolor=blue,
            ]{hyperref}

\include{header}

\author{Sebastian Posur}
\thanks{The author is supported by Deutsche Forschungsgemeinschaft (DFG) grant SFB-TRR 195: \emph{Symbolic Tools in Mathematics and their Application}}
\address{Department of mathematics, University of Siegen, 57068 Siegen, Germany}
\email{\href{mailto:Sebastian Posur <sebastian.posur@uni-siegen.de>}{sebastian.posur@uni-siegen.de}}

\begin{document}

\title[A constructive approach to Freyd categories]{A constructive approach to Freyd categories}

\begin{abstract}
In this paper
we give an algorithmic description of Freyd categories
that subsumes and enhances
the usual approach to finitely presented modules
in computer algebra.
The upshot is
a constructive approach to
finitely presented functors that only relies on a few basic algorithms.
\end{abstract}

\keywords{%
Freyd category,
finitely presented functor,
computable abelian category%
}
\subjclass[2010]{%
18E10, 
18E05, 
18A25, 
18E25, 
16S99
}
\maketitle

\tableofcontents

\input{FreydCats_content.tex}

\input{FreydCats.bbl}

\end{document}

%% file: header.tex
\usepackage[utf8]{inputenc} 
\usepackage[T1]{fontenc}
\usepackage{a4wide}
\usepackage{lmodern}
\usepackage[english]{babel}
\usepackage{mathrsfs}
\usepackage{etex}
\usepackage{mathtools}
\usepackage{latexsym}
\usepackage{amssymb}
\usepackage{amsthm}
\usepackage{amsmath}
\usepackage{caption}
\usepackage{mathabx}

\usepackage{colortbl}

\usepackage[all]{xy}
\usepackage{verbatim}
\usepackage{listings}
\usepackage{fancyvrb}

\usepackage{graphicx}

\usepackage[dvipsnames]{xcolor}
\usepackage{accents} 
\usepackage{enumerate}
\usepackage{wrapfig}
\usepackage{tikz}
\usepackage{tikz-cd}
\usetikzlibrary{automata,shapes,arrows,matrix,backgrounds,positioning,plotmarks,calc,patterns,matrix,decorations.pathreplacing,decorations.pathmorphing,decorations.text,decorations.markings}
\usepackage[colorinlistoftodos,shadow]{todonotes}

\usepackage{multirow}
\usepackage{mdwlist}

\usepackage{stmaryrd}
\usepackage{mathdots} 

\usepackage{fancyvrb}

\usepackage[toc,page]{appendix}
\usepackage{float}

\usepackage{extarrows}
\usepackage{pdflscape}
\usepackage{rotating}
\usepackage{etex}
\newtheoremstyle{mytheoremstyle} 
    {5pt}                    
    {5pt}                    
    {\itshape}                   
    {\parindent}                           
    {\bf}                   
    {.}                          
    {.5em}                       
    {}  

\theoremstyle{mytheoremstyle}

\newtheorem{theorem}{Theorem}[section]

\newtheorem{lemma}[theorem]{Lemma}

\newtheorem{corollary}[theorem]{Corollary}

\newtheoremstyle{mytdefintionstyle} 
    {5pt}                    
    {5pt}                    
    {\rm}                   
    {\parindent}                           
    {\bf}                   
    {.}                          
    {.5em}                       
    {}  

\theoremstyle{remark}
\newtheorem{remark}[theorem]{Remark}

\theoremstyle{mytdefintionstyle}
\newtheorem{definition}[theorem]{Definition}

\newtheorem{example}[theorem]{Example}

\newtheorem{construction}[theorem]{Construction}

\newtheorem*{notationnonumber}{Notation}

\newtheoremstyle{exmp_contd} 
{\topsep} {\topsep}%
{\upshape}
{}
{\bfseries}
{}
{ }
{\thmname{#1}\,\thmnumber{ #2}\thmnote{#3}\enspace(continued)}

\theoremstyle{exmp_contd}

\usepackage{xspace}
\input{pre.tex}

\tikzset{round left paren/.style={ncbar=0.5cm,out=120,in=-120}}
\tikzset{round right paren/.style={ncbar=0.5cm,out=60,in=-60}}
\newcolumntype{C}[1]{>{\centering\arraybackslash$}p{#1}<{$}}
\newlength{\mycolwd}
\usepackage{array, xcolor}
\definecolor{lightgray}{gray}{0.8}
\newcolumntype{L}{>{\raggedleft}p{0.28\textwidth}}
\newcolumntype{R}{p{0.8\textwidth}}

\definecolor{ctcolor}{gray}{0.95}

\definecolor{ctucolor}{gray}{0.85}

\makeatletter
\newcommand{\thickhline}{%
    \noalign {\ifnum 0=`}\fi \hrule height 1pt
    \futurelet \reserved@a \@xhline
}
\newcolumntype{"}{@{\hskip\tabcolsep\vrule width 1pt\hskip\tabcolsep}}
\makeatother

\usepackage{enumitem}
\newlist{theoremenumerate}{enumerate}{1}
\setlist[theoremenumerate]{label=(\arabic{theoremenumeratei}), ref=\thetheorem.(\arabic{theoremenumeratei}),noitemsep}

%% file: pre.tex
\definecolor{ExQ}{HTML}{0000FF}
\definecolor{Dec}{HTML}{E07B00}

\newcommand{\CapPkg}{\textsc{Cap}\xspace}

\newcommand{\AC}{\mathbf{A}}
\newcommand{\BC}{\mathbf{B}}
\newcommand{\XC}{\mathbf{X}}
\newcommand{\CC}{\mathbf{C}}

\newcommand{\PC}{\mathbf{P}}

\newcommand{\TC}{\mathbf{T}}
\newcommand{\firstproj}{ \begin{bmatrix}1\\0\end{bmatrix} }
\newcommand{\secondproj}{ \begin{bmatrix}0\\1\end{bmatrix} }
\newcommand{\bmatrow}[2]{ \begin{bmatrix}{#1} & {#2} \end{bmatrix} }
\newcommand{\pmatrow}[2]{ \begin{pmatrix}{#1} & {#2} \end{pmatrix} }

\newcommand{\N}{\mathbb{N}}
\newcommand{\Z}{\mathbb{Z}}

\newcommand{\Fr}{\mathrm{Fr}}

\newcommand{\Ab}{\mathbf{Ab}}

\newcommand{\op}{\mathrm{op}}

\newcommand{\Obj}{\mathrm{Obj}}

\newcommand{\Freyd}{\mathcal{A}}
\newcommand{\Xcal}{\mathcal{X}}
\newcommand{\Rcal}{\mathcal{R}}

\newcommand{\Tr}{\mathrm{Tr}}

\newcommand{\id}{\mathrm{id}}

\DeclareMathOperator{\End}{\mathrm{End}}
\DeclareMathOperator{\Hom}{\mathrm{Hom}}

\DeclareMathOperator{\kernel}{\mathrm{ker}}
\DeclareMathOperator{\ProjRes}{\mathrm{ProjRes}}

\DeclareMathOperator{\cokernel}{\mathrm{coker}}

\DeclareMathOperator{\image}{\mathrm{im}}

\DeclareMathOperator{\Tor}{\mathrm{Tor}}
\DeclareMathOperator{\Ext}{\mathrm{Ext}}

\DeclareMathOperator{\B}{\mathbf{B}}

\newcommand{\fpmodl}{\text{-}\mathrm{fpmod}}
\newcommand{\fpmodr}{\mathrm{fpmod}\text{-}}

\newcommand{\fpgrmodl}{\text{-}\mathrm{fpgrmod}}
\newcommand{\fpgrmodr}{\mathrm{fpgrmod}\text{-}}

\newcommand{\Rows}{\mathrm{Rows}}
\newcommand{\Cols}{\mathrm{Cols}}
\newcommand{\grRows}{\mathrm{grRows}}
\newcommand{\grCols}{\mathrm{grCols}}

\newcommand{\fp}{\mathrm{fp}}
\newcommand{\grModl}{\text{-}\mathrm{grMod}}
\newcommand{\grModr}{\mathrm{grMod}\text{-}}

\newcommand{\Modl}{\text{-}\mathrm{Mod}}
\newcommand{\Modr}{\mathrm{Mod}\text{-}}

\newcommand{\GAP}{\texttt{GAP}\xspace}

\newcommand{\KernelEmbedding}{\mathrm{KernelEmbedding}}

\newcommand{\CokernelProjection}{\mathrm{CokernelProjection}}

\newcommand{\operation}{operation\xspace}
\newcommand{\operations}{operations\xspace}
\newcommand{\constructs}{constructs\xspace}
\newcommand{\constructing}{constructing\xspace}

%% file: FreydCats_content.tex
\section{Introduction}
With this paper we hope to convince the reader that
important parts of category theory such as the theory of Freyd categories
are \emph{inherently algorithmic}.
To an additive category $\PC$, Peter Freyd associated the so-called Freyd category $\mathcal{A}( \PC )$ \cite{FreydRep, BelFredCats}
that equips $\PC$ with cokernels in a universal way.
If we think
of objects and morphisms in Freyd categories as data types,
then theorems like the existence of kernels in $\Freyd( \PC )$
(assuming $\PC$ has weak kernels)
can actually be proven by
providing explicit constructions.
Such constructions
can in turn be directly implemented
in computer algebra projects like \CapPkg (Categories, Algorithms, Programming)
\cite{CAP-project, GutscheDoktor, PosurDoktor} for performing effective computations.
In this paper we provide various important constructions for Freyd categories.

Freyd categories have already played an important hidden role
in computer algebra systems.
A common data structure for finitely presented (left) modules
over a ring $R$ in computer algebra systems like
\textsc{Singular} \cite{singular410}, \textsc{Macaulay2} \cite{M2},
or in a software project like \texttt{homalg} \cite{homalg-project} 
written in \GAP \cite{GAP4}
is given by matrices over $R$,
where an $m \times n$-matrix $M$ is interpreted as the cokernel
of its induced map between free row modules $R^{1 \times m} \stackrel{M}{\longrightarrow} R^{1 \times n}$.
Performing operations like taking kernels
in terms of this data structure
can be seen as 
a special instance
of performing those operations
within a Freyd category,
namely $\Freyd( \Rows_R)$,
the Freyd category associated to the additive category of row modules.
We make this hidden role played by Freyd categories explicit, 
and push it forward to reach new applications for computer algebra like the following computations with finitely presented functors,
i.e., functors that are given as the cokernels of a natural transformation between representable functors:
\begin{itemize}
 \item Determination of sets of natural transformations
between two finitely presented functors like $\Tor_1(M,-)$ and $\Ext^i(A,-)$,
where $M$ and $A$ are suitable $R$-modules (see Subsection \ref{subsection:nat}).
\item Construction of injective resolutions of finitely presented functors (Subsection \ref{subsection:inj}).
\item Deciding whether a given finitely presented functor is left exact (Subsection \ref{subsection:left})
or right exact (Subsection \ref{subsection:right}).
\end{itemize}

This paper is organized as follows.
In Section \ref{section:computability} we explain our constructive approach to category theory.
In Section \ref{section:constructive_freyd_categories} 
we give a constructive proof of the main theorem by Freyd \cite{FreydRep}
in a way such that a direct computer implementation becomes possible:
\begin{center}
Given an additive category $\PC$,
then $\Freyd( \PC )$ is abelian if and only
if $\PC$ has weak kernels. 
\end{center}
To this end, we provide explicit constructions for sufficiently many operations in $\Freyd( \PC )$
like taking (co)kernels and computing lifts/colifts along monomorphisms/epimorphisms.

In Section \ref{section:applications_modules} 
we give several interpretations of Freyd categories.
The Freyd category $\Freyd( \Rows_R )$ provides a model for the category of finitely presented (left) modules over $R$ (see Example \ref{example:fpmod})
and a similar result holds for finitely presented graded modules (see Example \ref{example:fpgrmod}).
Moreover, we characterize so-called (left) computable rings (introduced by Barakat and Lange-Hegermann in \cite{BL})
as those rings $R$ for which $\Freyd( \Rows_R )$ is abelian with decidable equality for morphisms.

Propositions that 
classically can be taken for granted
have to be explicitly realized by
algorithms in our constructive setup.
In Section \ref{section:undecidable}
(and also in the end of Section \ref{section:lifts})
we provide examples of rings with a curious behavior from
a computational point of view:
\begin{itemize}
 \item A ring with decidable equality, but the existence of a particular solution
       of left- and right-sided linear systems are computationally undecidable
       (Subsection \ref{subsection:computable_lift_computable_colift}).
 \item A ring $P$ with decidable equality such that we can find particular solutions of left-sided linear systems,
       but the existence of a particular solution of right-sided linear systems is computationally undecidable
       (Subsection \ref{subsection:computable_lift_undecidable_colift}).
       Such a ring gives us an example of
       an iterated Freyd category $\Freyd(\Freyd( \Rows_P )^{\op})$ that is abelian
       but its equality for morphisms is computationally undecidable (see Subsection \ref{subsection:add_cat_dec_eq_kernels_undec_lifts}).
 \item A commutative ring $R$ that is not coherent (it has a finitely generated ideal that is not finitely presented),
       but we can find particular solutions of linear systems
       and the iterated Freyd category $\Freyd(\Freyd( \Rows_R )^{\op})$ is abelian
       with decidable equality for morphisms (see Theorem \ref{theorem:ring_non_coherent_but_iterated_freyd_computable}).
\end{itemize}

In Section \ref{section:lifts}
we study additive categories $\PC$
equipped with homomorphism structures,
where a homomorphism structure is a 
powerful tool that allows us to
solve linear systems in $\PC$ in the sense of Definition \ref{definition:linear_system}
(see Theorem \ref{theorem:linear_system_in_P} for our solving strategy).
This tool is also used in the determination of
sets of natural transformations
between finitely presented functors,
with which we finally deal in the last Section \ref{section:applications_fpfunctors}.

It is planned to implement
a constructor for the Freyd category
of an additive category
as part of the \CapPkg project.
\CapPkg is a software project written in \GAP and supports the programmer in the implementation of category theory based constructions.
In particular, \CapPkg will make it fairly easy to realize the constructive methods for
finitely presented functors described in Section \ref{section:applications_fpfunctors}
on the computer.
A first \CapPkg-based attempt of such an implementation
building on the ideas of this paper
was done by Bies in \cite{BiesCAPPresentationCategory}
and used in \cite{Bies2017fam} in the context of sheaf cohomology computations over toric varieties.

\begin{notationnonumber} 
We write morphisms between direct sums $A \oplus B \rightarrow C \oplus D$
in additive categories as matrices using the row convention
$\begin{pmatrix}
  \gamma_{AC} & \gamma_{AD} \\
  \gamma_{BC} & \gamma_{BD}
 \end{pmatrix}
$
for morphisms $\gamma_{AC}: A \rightarrow C$, $\gamma_{AD}: A \rightarrow D$,
$\gamma_{BC}: B \rightarrow C$, $\gamma_{BD}: B \rightarrow D$.
We prefer writing $\alpha \cdot \beta: A \rightarrow C$ to $\beta \circ \alpha$
for the composition of morphisms $\alpha: A \rightarrow B$ and $\beta: B \rightarrow C$,
since this matches the row convention in a way that composition of morphisms is simply given by matrix multiplication.
\end{notationnonumber}

\section{Constructive category theory}\label{section:computability}
To present our algorithmic approach to Freyd categories,
we chose the language of constructive
mathematics (see, e.g., \cite{MRRConstructiveAlgebra}).
We did that
for the following reasons:
the language of constructive mathematics
\begin{enumerate}
 \item reveals the algorithmic content of the theory of Freyd categories,
 \item is perfectly suited for describing generic algorithms,
 i.e., constructions not depending on particular choices of data structures,
 \item allows us to express our algorithmic ideas without choosing some particular
 model of computation (like Turing machines),
 \item encompasses classical mathematics,
 i.e., all results stated in constructive mathematics are also valid classically,
 \item does not differ very much from the classical language in our particular setup.
\end{enumerate}
In constructive mathematics
the notions of data types and algorithms (or operations)
are taken as primitives
and every property must have an algorithmic interpretation.
For example given an additive category $\AC$ 
we interpret the property 
\begin{center}
$\AC$ has kernels 
\end{center}
as follows: we have algorithms that compute for given
\begin{itemize}
  \item $A,B \in \Obj_{\AC}$, $\alpha \in \Hom_{\AC}(A,B)$ an object $\kernel( \alpha ) \in \Obj_{\AC}$
        and a morphism 
        \[\KernelEmbedding( \alpha ) \in \Hom_{\AC}(\kernel( \alpha ), A)\]
        for which $\KernelEmbedding( \alpha ) \cdot \alpha = 0$,
  \item $A,B,T \in \Obj_{\AC}$, $\alpha \in \Hom_{\CC}(A,B)$, $\tau \in \Hom_{\AC}(T,A)$
        such that $\tau \cdot \alpha = 0$
        a morphism $u \in \Hom_{\AC}(T,\kernel( \alpha ))$ 
        such that 
        \[ u \cdot \KernelEmbedding( \alpha ) = \tau,\]
        where $u$ is uniquely determined (up to $=$) by this property.
\end{itemize}
Another important example is given by \emph{decidable equality},
where we interpret the property that for all objects $A,B \in \AC$, we have
\[
 \forall \alpha, \beta \in \Hom_{\AC}(A,B): (\alpha = \beta) \vee (\alpha \neq \beta)
\]
as follows: we are given an algorithm
that decides or disproves equality of a given pair of morphisms.
In the appendix \ref{appendixa}
we enlisted a constructive interpretation
of various kinds of categories, e.g., additive or abelian categories
(cf.~the corresponding list in \cite[Appendix B]{BL_GabrielMorphisms}
and note the difference in our treatment of equalities for morphisms explained in Remark \ref{remark:treatment_of_equality}).

On the other hand, we allow ourselves to work classically
whenever we interpret Freyd categories
in terms of finitely presented functors
(this happens in Section \ref{section:applications_fpfunctors}).
The reason for this is pragmatic:
we want to demonstrate the usefulness of having Freyd categories
computationally available, and we believe that this can be done by 
interpreting Freyd categories in terms of other categories that classical mathematicians care about.

\section{Constructive Freyd categories}\label{section:constructive_freyd_categories}
We recall the definition of a Freyd category.

\begin{definition}
 Let $\PC$ be an additive category.
 The \textbf{Freyd category} $\mathcal{A}(\PC)$ is given 
 by the following data:
 \begin{enumerate}
  \item An object in $\mathcal{A}(\PC)$ is simply a 
  morphism in $\PC$.
  We will write such an object as $(A \stackrel{\rho_A}{\longleftarrow} R_A)$, even though $R_A$ and $\rho_A$
  do not formally depend\footnote{The symbols $R_A$ and $\rho_A$ stand for \emph{relations for $A$}. Thus,
  we think of a morphism $(A \stackrel{\rho_A}{\longleftarrow} R_A)$ as an additional datum for its range $A$.} on $A$.
  \item A morphism in $\mathcal{A}(\PC)$ from $(A \stackrel{\rho_A}{\longleftarrow} R_A)$ to $(B \stackrel{\rho_B}{\longleftarrow} R_B)$ 
        is given by a morphism 
        $A \stackrel{\alpha}{\longrightarrow} B$ in $\PC$
        such that there exists another morphism $R_A \stackrel{\rho_{\alpha}}{\longrightarrow} R_B$  
        rendering the diagram
        \begin{center}
         \begin{tikzpicture}[label/.style={postaction={
          decorate,
          decoration={markings, mark=at position .5 with \node #1;}},
          mylabel/.style={thick, draw=none, align=center, minimum width=0.5cm, minimum height=0.5cm,fill=white}}]
          \coordinate (r) at (2.5,0);
          \coordinate (u) at (0,2);
          \node (A) {$A$};
          \node (B) at ($(A)+(r)$) {$R_A$};
          \node (C) at ($(A) - (u)$) {$B$};
          \node (D) at ($(B) - (u)$) {$R_B$};
          \draw[->,thick] (B) --node[above]{$\rho_A$} (A);
          \draw[->,thick] (D) --node[above]{$\rho_B$} (C);
          \draw[->,thick] (A) --node[left]{$\alpha$} (C);
          \draw[->,thick,dashed] (B) --node[right]{$\rho_{\alpha}$} (D);
          \end{tikzpicture}
        \end{center}
        commutative. 
        We call $\alpha$ the \textbf{morphism datum} and any such $\rho_{\alpha}$ a \textbf{morphism witness}.
        We often write $\{\alpha, \rho_{\alpha}\}$ for a morphism in $\mathcal{A}(\PC)$ 
        to highlight a particular choice of a morphism witness $\rho_{\alpha}$ for a morphism datum $\alpha$.
        We define two morphisms
        $A \stackrel{\alpha}{\longrightarrow} B$, $A \stackrel{\alpha'}{\longrightarrow} B$ 
        from $(A \stackrel{\rho_A}{\longleftarrow} R_A)$ to $(B \stackrel{\rho_B}{\longleftarrow} R_B)$
        to be \textbf{equal (in $\Freyd(\PC)$)} if there exists a lift $\lambda$ of the following
        diagram:
        \begin{center}
        \begin{tikzpicture}[label/.style={postaction={
          decorate,
          decoration={markings, mark=at position .5 with \node #1;}},
          mylabel/.style={thick, draw=none, align=center, minimum width=0.5cm, minimum height=0.5cm,fill=white}}]
          \coordinate (r) at (3.5,0);
          \coordinate (u) at (0,2);
          \node (m) {$B$};
          \node (n) at ($(m)+(r)$) {$R_B$.};
          \node (r1) at ($(m) + (u)$) {$A$};
          \draw[->,thick] (n) --node[above]{$\rho_B$} (m);
          \draw[->,thick,dashed] (r1) --node[above]{$\lambda$} (n);
          \draw[->,thick] (r1) --node[left]{$\alpha - \alpha'$} (m);
    \end{tikzpicture}
  \end{center}
  We call any such $\lambda$ a \textbf{witness for $\alpha$ and $\alpha'$ being equal}.
  \item Composition and identities are directly inherited from $\PC$.
 \end{enumerate}
\end{definition}

\begin{remark}
Being equal in $\Freyd( \PC )$ clearly defines an equivalence relation (transitivity corresponds to addition of witnesses)
that is compatible with composition.
Furthermore, $\Freyd( \PC )$ inherits the structure of an additive category
from the category of arrows of $\PC$.
\end{remark}

\begin{remark}\label{remark:constructiveness}
 Since we are interested in a constructive approach to Freyd
 categories, we say a few words about witnesses.
 Classically, morphisms in the Freyd category
 are the equivalence classes of ``being equal in $\Freyd( \PC )$''.
 For us, however, it is more convenient to use the language
 of morphism data and witnesses
 since these are the actual entities for which
 we define our algorithms (see Constructions \ref{construction:cokernel}, \ref{construction:kernel}, \ref{construction:lift_along_monos}, \ref{construction:colift_along_epis}).
 In particular, we avoid unnecessary picking of representatives
 in the beginning of all these constructions.
\end{remark}

Stating the main theorem for $\mathcal{A}(\PC)$
requires another definition.

\begin{definition}\label{definition:weak_kernels}
 Let $\PC$ be an additive category. For a morphism $\alpha: A \rightarrow B$ in $\PC$,
 a \textbf{weak kernel} consists the following data:
 \begin{theoremenumerate}
  \item An object $K \in \PC$.
  \item A morphism $\kappa: K \rightarrow A$ such that $\kappa \cdot \alpha  = 0$.
  \item An \operation that \constructs for $T \in \PC$ and test morphism $\tau: T \rightarrow A$ such that $\tau \cdot \alpha = 0$
        a morphism $u(\tau): T \rightarrow K$ with $u(\tau) \cdot \kappa = \tau$.\label{definition:weak_kernels_3} 
        \begin{center}
   \begin{tikzpicture}[label/.style={postaction={
          decorate,
          decoration={markings, mark=at position .5 with \node #1;}},
          mylabel/.style={thick, draw=none, align=center, minimum width=0.5cm, minimum height=0.5cm,fill=white}}]
          \coordinate (r) at (2.5,0);
          \coordinate (u) at (0,1.5);
          \node (K) {$K$};
          \node (A) at ($(K)+(r) - 0.5*(u)$) {$A$};
          \node (T) at ($(K) - (u)$) {$T$};
          \draw[->,thick] (K) --node[above]{$\kappa$} (A);
          \draw[->,thick,dashed] (T) --node[left]{$u(\tau)$} (K);
          \draw[->,thick] (T) --node[below]{$\tau$} (A);
    \end{tikzpicture}
  \end{center}
 \end{theoremenumerate}
 Note that the morphism $u(\tau)$ \emph{does not have to be unique} with this property,
 so we really weaken the usual definition of a kernel.
 Furthermore, we say $\PC$ \textbf{has weak kernels}
 if it comes equipped with an \operation \constructing the triple $(K,\kappa,u)$
 for given $\alpha$.
\end{definition}

The following main theorem about $\Freyd( \PC )$ is due to Freyd \cite{FreydRep}.
\begin{theorem}\label{theorem:freyd}
 Let $\PC$ be an additive category.
 Then $\mathcal{A}(\PC)$ is an abelian category
 if and only if $\PC$ has weak kernels.
\end{theorem}

Freyd gives two proofs of his theorem:
a non-elementary one using functor categories
\cite[Proposition 1.4]{FreydRep}
and a more elementary one using an auxiliary theorem
for recognizing abelian categories via the existence
of factorizations of morphisms into cokernel projections
and kernel embeddings \cite[Section 3]{FreydRep}.
Our goal is to render Theorem \ref{theorem:freyd} constructive
in a way such that a direct computer implementation becomes possible.
To this end, we will state and prove the explicit construction steps
for (co)kernels (and their universal properties) in $\mathcal{A}(\PC)$, as well as
for lifts along monomorphisms and colifts along epimorphisms.
Our explicitness reveals what is otherwise
hidden in propositions, namely the important role of witnesses
for the construction of morphism data.

\subsection{Cokernels in Freyd categories}
We start with the construction of cokernels.

\begin{construction}[Cokernels]\label{construction:cokernel}
 Given a morphism
 \[
  \{\alpha, \rho_{\alpha}\}: (A \stackrel{\rho_A}{\longleftarrow} R_A) \longrightarrow (B \stackrel{\rho_B}{\longleftarrow} R_B)
 \]
 in $\mathcal{A}(\PC)$, the following diagram
 shows us how to construct its cokernel projection along with the universal property:
 \begin{center}
    \begin{tikzpicture}[label/.style={postaction={
          decorate,
          decoration={markings, mark=at position .5 with \node #1;}},
          mylabel/.style={thick, draw=none, align=center, minimum width=0.5cm, minimum height=0.5cm,fill=white}}]
          \coordinate (r) at (6,0);
          \coordinate (u) at (0,2);
          \node (A) {$(A \stackrel{\rho_A}{\longleftarrow} R_A)$};
          \node (B) at ($(A)+0.8*(r)$) {$(B \stackrel{\rho_B}{\longleftarrow} R_B)$};
          \node (C) at ($(B) + (r) + (u)$) {$(B \stackrel{\begin{pmatrix}\rho_B \\ \alpha\end{pmatrix}}{\longleftarrow} R_B \oplus A)$};
          \node (T) at ($(B) + (r) - (u)$) {$(T \stackrel{\rho_T}{\longleftarrow} R_T)$.};
          \draw[->,thick] (A) --node[above]{$\{\alpha, \rho_{\alpha}\}$} (B);
          \draw[->,thick] (B) --node[below,yshift=-0.1em,xshift=-1em]{$\{\tau, \rho_{\tau}\}$} (T);
          \draw[->,thick] (B) --node[above,xshift=-1.5em,yshift=0.5em]{$\{\id_B, \begin{pmatrix}\id_{R_B} & 0\end{pmatrix}\}$} (C);
          \draw[->,thick, dashed] (C) --node[right]{$\{\tau, \begin{pmatrix}\rho_{\tau} \\ \sigma \end{pmatrix}\}$} (T);
          \node (Ae) at ($(A)+(-0.75,-0.15)$) {};
          \node (Te) at ($(T)-(-0.4,0.15)$) {};
          \draw[bend right,->,thick,dotted,label={[above]{$\sigma$}},out=360-25,in=360-155] (Ae) to (Te);
    \end{tikzpicture}
 \end{center}
 How to read this diagram: the solid arrow pointing up right is the cokernel projection,
 the solid arrow pointing down right is a test morphism for the universal property of the cokernel,
 and the dashed arrow pointing straight down is the morphism induced by the universal property.
 The dotted arrow labeled with $\sigma$ is a witness for the composition $\{\alpha, \rho_\alpha\} \cdot \{\tau,\rho_{\tau}\}$ being zero,
 i.e., it denotes a morphism $\sigma: A \rightarrow R_T$ such that $\sigma \cdot \rho_T = \alpha \cdot \tau$.
 We see that $\sigma$ is used in the construction of the morphism witness for the induced morphism,
 but not in any morphism data.
 We also see that no morphism witness is used for the construction of any morphism data.
\end{construction}
\begin{proof}[Correctness of the construction]
 It is an easy calculation that all morphisms are well-defined.
 A witness for the composition $\{\alpha, \rho_\alpha\} \cdot \{\id_B, \begin{pmatrix}\id_{R_B} & 0\end{pmatrix}\}$ being zero
 is given by the natural inclusion $A \hookrightarrow R_B \oplus A$.
 The commutativity of the triangle even holds strictly in the category of arrows of $\PC$.
 For the uniqueness of the induced morphism, it suffices to see that our construction of
 the cokernel projection is an epimorphism, which follows from the next Lemma \ref{lemma:epis}.
\end{proof}

\begin{lemma}\label{lemma:epis}
 Every morphism in $\mathcal{A}(\PC)$ of the form
 \[
  \{\id_A, \rho\}: (A \stackrel{\rho_1}{\longleftarrow} R_1) \longrightarrow (A \stackrel{\rho_2}{\longleftarrow} R_2)
 \]
 is an epimorphism.
\end{lemma}
\begin{proof}
 See also \cite[first part of the proof of Lemma 3.2.1]{FreydRep}.
 Given a test morphism for $\{\id_A, \rho\}$ being an epimorphism, i.e., a morphism
 \[
 \{\beta, \rho_{\beta}\}: (A \stackrel{\rho_2}{\longleftarrow} R_2) \longrightarrow (B \stackrel{\rho_B}{\longleftarrow} R_B)
 \]
 such that $\{\id_A, \rho\} \cdot \{\beta, \rho_{\beta}\} = 0$,
 we can take any witness $\sigma: A \rightarrow R_B$ for this composition being zero
 as a witness for $\{\beta, \rho_{\beta}\}$ being zero.
\end{proof}

\subsection{Kernels in Freyd categories}

For the construction of kernels in $\mathcal{A}(\PC)$,
we first introduce weak pullbacks in $\PC$
(mainly for introducing our notation).

\begin{definition}\label{definition:weak_pullbacks}
 Let $\PC$ be an additive category. For a given cospan $A \stackrel{\alpha}{\longrightarrow} B \stackrel{\gamma}{\longleftarrow} C$ in $\PC$,
 a \textbf{weak pullback} consists the following data:
 \begin{enumerate}
  \item An object $A \times_B C \in \PC$.
  \item Morphisms $\firstproj: A \times_B C \rightarrow A$ and $\secondproj: A \times_B C \rightarrow C$
        such that $\firstproj \cdot \alpha = \secondproj \cdot \gamma$.
  \item An \operation that \constructs for $T \in \PC$ and morphisms $p: T \rightarrow A$, $q: T \rightarrow C$ such that $p \cdot \alpha = q \cdot \gamma$
        a morphism $\bmatrow{p}{q}: T \rightarrow A \times_B C$ satisfying
        \begin{center}
        \begin{tabular}{ccccc}
         $p = \bmatrow{p}{q} \cdot \firstproj$
         & & and & &
         $q = \bmatrow{p}{q} \cdot \secondproj$.
        \end{tabular} 
        \end{center}
 \end{enumerate}
 Note that the morphism $\bmatrow{p}{q}$ \emph{does not have to be unique} with this property,
 so we really weaken the usual definition of a pullback.
 Also note that we use \emph{angular} matrices for weak pullback morphisms (in contrast to round matrices for direct sum morphisms).
\end{definition}

\begin{remark}
 If $\PC$ has weak kernels, then it also has weak pullbacks, since we can compute weak pullbacks
 from weak kernels and direct sums
 in the same way as we can compute pullbacks from kernels and direct sums.
\end{remark}

\begin{construction}[Kernels]\label{construction:kernel}
 Given a morphism
 \[
  \{\alpha, \rho_{\alpha}\}: (A \stackrel{\rho_A}{\longleftarrow} R_A) \longrightarrow (B \stackrel{\rho_B}{\longleftarrow} R_B)
 \]
 in $\mathcal{A}(\PC)$, the following diagram
 shows us how to construct its kernel embedding along with the universal property:
 \begin{center}
    \begin{tikzpicture}[label/.style={postaction={
          decorate,
          decoration={markings, mark=at position .5 with \node #1;}},
          mylabel/.style={thick, draw=black, align=center, minimum width=0.5cm, minimum height=0.5cm,fill=white}}]
          \coordinate (r) at (5.5,0);
          \coordinate (u) at (0,2);
          \node (A) {$(A \stackrel{\rho_A}{\longleftarrow} R_A)$};
          \node (B) at ($(A)+0.8*(r)$) {$(B \stackrel{\rho_B}{\longleftarrow} R_B)$.};
          \node (K) at ($(A) - 1.3*(r) + (u)$) {$\big(R_B \times_B A \stackrel{\firstproj}{\longleftarrow} (R_B \times_B A) \times_A R_A\big)$};
          \node (T) at ($(A) - 1.3*(r) - (u)$) {$(T \stackrel{\rho_T}{\longleftarrow} R_T)$};
          \draw[->,thick] (A) --node[above]{$\{\alpha, \rho_{\alpha}\}$} (B);
          \draw[->,thick] (T) --node[below,yshift=-0.2em]{$\{\tau, \rho_{\tau}\}$} (A);
          \draw[->,thick] (K) --node[above,xshift=0em,yshift=0.3em]{$\{\secondproj, \secondproj\}$} (A);
          \draw[->,thick,dashed,label={[mylabel]{$\big\{\bmatrow{\sigma}{\tau},\bmatrow{\rho_T \cdot \bmatrow{\sigma}{\tau}}{\rho_{\tau}}\big\}$}}] (T) -- (K);
          
          \node (Be) at ($(B)-(-0.4,0.15)$) {};
          \node (Te) at ($(T)+(-0.75,-0.15)$) {};
          \draw[bend right,->,thick,dotted,label={[above]{$\sigma$}},out=360-25,in=360-155] (Te) to (Be);
    \end{tikzpicture}
 \end{center}
 How to read this diagram: 
 the occurring weak pullbacks are defined by
 \begin{center}
 \begin{tabular}{ccccc}
     \begin{tikzpicture}[label/.style={postaction={
          decorate,
          decoration={markings, mark=at position .5 with \node #1;}},
          mylabel/.style={thick, draw=none, align=center, minimum width=0.5cm, minimum height=0.5cm,fill=white}},
          baseline = (base)]
          \coordinate (r) at (2.5,0);
          \coordinate (u) at (0,2);
          \node (A) {$R_B \times_B A$};
          \node (B) at ($(A)+(r)$) {$A$};
          \node (C) at ($(A) - (u)$) {$R_B$};
          \node (D) at ($(B) - (u)$) {$B$};
          \node (base) at ($0.5*(A) + 0.5*(C)$) {};
          \draw[->,thick] (A) -- (B);
          \draw[->,thick] (C) --node[above]{$\rho_B$} (D);
          \draw[->,thick] (A) -- (C);
          \draw[->,thick] (B) --node[right]{${\alpha}$} (D);
          \end{tikzpicture}
  
        & & and & &
  
         \begin{tikzpicture}[label/.style={postaction={
          decorate,
          decoration={markings, mark=at position .5 with \node #1;}},
          mylabel/.style={thick, draw=none, align=center, minimum width=0.5cm, minimum height=0.5cm,fill=white}},
          baseline = (base)]
          \coordinate (r) at (3.5,0);
          \coordinate (u) at (0,2);
          \node (A) {$(R_B \times_B A) \times_A R_A$};
          \node (B) at ($(A)+(r)$) {$R_A$};
          \node (C) at ($(A) - (u)$) {$R_B \times_B A$};
          \node (D) at ($(B) - (u)$) {$A$.};
          \node (base) at ($0.5*(A) + 0.5*(C)$) {};
          \draw[->,thick] (A) -- (B);
          \draw[->,thick] (C) --node[above]{$\secondproj$} (D);
          \draw[->,thick] (A) -- (C);
          \draw[->,thick] (B) --node[right]{${\rho_A}$} (D);
          \end{tikzpicture}
  \end{tabular}
 \end{center}
 The solid arrow pointing down right is the kernel embedding,
 the solid arrow pointing up right is a test morphism for the universal property of the kernel,
 and the dashed arrow pointing straight up is the morphism induced by the universal property.
 The dotted arrow depicts $\sigma: T \rightarrow R_B$, a witness for the composition
 $\{\tau,\rho_{\tau}\} \cdot \{\alpha, \rho_\alpha\}$ being zero.
 We see that no morphism witness is used for the construction of any morphism data.
 But now, in contrast to the cokernel construction,
 the witness $\sigma$ \emph{is} used in the construction of the morphism datum of the induced morphism.
\end{construction}
\begin{proof}[Correctness of the construction]
 It is easy to see that all morphisms are well-defined.
 A witness for the composition $\{\secondproj, \secondproj\}\cdot \{\alpha, \rho_\alpha\}$ being zero
 is given by the projection $\firstproj: R_B\times_B A \rightarrow R_B$.
 The commutativity of the triangle even holds strictly in the category of arrows of $\PC$.
 For the uniqueness of the induced morphism, it suffices to see that our construction of
 the kernel embedding is a monomorphism, which follows from the next Lemma \ref{lemma:monos}.
\end{proof}

\begin{lemma}\label{lemma:monos}
 Given a cospan $A \stackrel{\alpha}{\longrightarrow} B \stackrel{\gamma}{\longleftarrow} C$ in $\PC$,
 the morphism in $\mathcal{A}(\PC)$ defined by
 \[
  \{\alpha, \secondproj\}: (A \stackrel{\firstproj}{\longleftarrow} A\times_B C) \longrightarrow (B \stackrel{\gamma}{\longleftarrow} C)
 \]
 is a monomorphism.
\end{lemma}
\begin{proof}
 See also \cite[first part of proof of Lemma 3.2.2]{FreydRep}.
 Given a test morphism for $\{\alpha, \secondproj\}$ being a monomorphism, i.e., a morphism
 \[
 \{\delta, \rho_{\delta}\}: (D \stackrel{\rho_D}{\longleftarrow} R_D) \longrightarrow (A \stackrel{\firstproj}{\longleftarrow} A\times_B C)
 \]
 such that $\{\delta, \rho_{\delta}\} \cdot \{\alpha, \secondproj\} = 0$,
 take a witness $\sigma: D \rightarrow C$ for this composition being zero.
 Then $\bmatrow{\delta}{\sigma}: D \rightarrow A \times_B C$ is a witness for $\{\delta, \rho_{\delta}\}$ being zero.
\end{proof}

\subsection{Lift along monomorphisms in Freyd categories}

One axiom of abelian categories states that
every monomorphism is the kernel of its cokernel.
When we rephrase this axiom constructively, 
we see that we have to be able to construct lifts along
monomorphisms.

\begin{definition}\label{definition:computable_lifts}
 A category $\AC$ \textbf{has decidable lifts}
 if we have an algorithm
 that creates for a given 
 cospan $A \stackrel{\alpha}{\longrightarrow} B \stackrel{\gamma}{\longleftarrow} C$ in $\AC$
 a morphism $\lambda: A \rightarrow C$ satisfying $\lambda \cdot \gamma = \alpha$ (called a \textbf{lift} of $\alpha$ along $\gamma$)
 or disproves its existence.
 If $\gamma$ is a monomorphism, we also call $\lambda$ a \textbf{lift along a monomorphism}.
 
 Dually, we say $\AC$ \textbf{has decidable colifts}
 if we have an algorithm
 that creates for a given
 span $A \stackrel{\alpha}{\longleftarrow} B \stackrel{\gamma}{\longrightarrow} C$ in $\AC$
 a morphism $\lambda: C \rightarrow A$ satisfying $\gamma \cdot \lambda = \alpha$ (called a \textbf{colift} of $\alpha$ along $\gamma$)
 or disproves its existence.
 If $\gamma$ is an epimorphism, we also call $\lambda$ a \textbf{colift along an epimorphism}.
\end{definition}

\begin{construction}[Lifts along monomorphisms]\label{construction:lift_along_monos}
 The kernel embedding of a given monomorphism
 \[
  \{\alpha, \rho_{\alpha}\}: (A \stackrel{\rho_A}{\longleftarrow} R_A) \longrightarrow (B \stackrel{\rho_B}{\longleftarrow} R_B)
 \]
 in $\mathcal{A}(\PC)$ is zero.
 A witness of this fact 
 is given by a morphism $\sigma: R_B \times_B A \rightarrow R_A$ such that
 \begin{equation}\label{equation:witness}
  \secondproj = \sigma \cdot \rho_A
 \end{equation}
 (see Construction \ref{construction:kernel}).
 Now, the following diagram shows us how to construct a lift along $\{\alpha, \rho_{\alpha}\}$
 for a given test morphism:
 \begin{center}
    \begin{tikzpicture}[label/.style={postaction={
          decorate,
          decoration={markings, mark=at position .5 with \node #1;}},
          mylabel/.style={thick, draw=black, align=center, minimum width=0.5cm, minimum height=0.5cm,fill=white}}]
          \coordinate (r) at (5.5,0);
          \coordinate (u) at (0,2);
          \node (A) {$(B \stackrel{\rho_B}{\longleftarrow} R_B)$};
          \node (B) at ($(A)+1.09*(r)$) {$(B \smash{\stackrel{\begin{pmatrix}\rho_B \\ \alpha\end{pmatrix}}{\longleftarrow}} R_B \oplus A)$.};
          \node (K) at ($(A) - 1*(r) + (u)$) {$(A \stackrel{\rho_A}{\longleftarrow} R_A)$};
          \node (T) at ($(A) - 1*(r) - (u)$) {$(T \stackrel{\rho_T}{\longleftarrow} R_T)$};
          \draw[->,thick] (A) --node[above]{$\{\id_B, \begin{pmatrix}\id_{R_B} & 0\end{pmatrix}\}$} (B);
          \draw[->,thick] (T) --node[below,yshift=-0.2em]{$\{\tau, \rho_{\tau}\}$} (A);
          \draw[->,thick] (K) --node[above,xshift=0em,yshift=0.3em]{$\{\alpha, \rho_{\alpha}\}$} (A);
          \draw[->,thick,dashed,label={[mylabel]{$\big\{\tau_A, \bmatrow{\rho_{\tau}-\rho_T \cdot \tau_{R_B}}{\rho_T \cdot \tau_A}\cdot \sigma \big\}$}}] (T) -- (K);
          
          \node (Be) at ($(B)-(-0.4,0.15)$) {};
          \node (Te) at ($(T)+(-0.75,-0.15)$) {};
          \draw[bend right,->,thick,dotted,label={[mylabel]{$\pmatrow{\tau_{R_B}}{\tau_A}$}},out=360-25,in=360-155] (Te) to (Be);
    \end{tikzpicture}
 \end{center}
 How to read this diagram: the solid horizontal arrow is the cokernel projection of 
 our monomorphism $\{\alpha, \rho_{\alpha}\}$ (see Construction \ref{construction:cokernel}).
 The dotted arrow is a witness for the composition of the test morphism $\{\tau, \rho_{\tau}\}$ with the cokernel projection being zero.
 The upwards pointing dashed arrow is the desired lift,
 whose morphism witness involves the weak pullback induced morphism given by the diagram
 \begin{center}
  \begin{tikzpicture}[label/.style={postaction={
          decorate,
          decoration={markings, mark=at position .5 with \node #1;}},
          mylabel/.style={thick, draw=none, align=center, minimum width=0.5cm, minimum height=0.5cm,fill=white}},
          baseline = (base)]
          \coordinate (r) at (2.5,0);
          \coordinate (u) at (0,2);
          \node (A) {$R_B \times_B A$};
          \node (B) at ($(A)+(r)$) {$A$};
          \node (C) at ($(A) - (u)$) {$R_B$};
          \node (D) at ($(B) - (u)$) {$B$.};
          \node (T) at ($(A) - (r) + 0.5*(u)$) {$R_T$};
          \node (base) at ($0.5*(A) + 0.5*(C)$) {};
          \draw[->,thick] (A) -- (B);
          \draw[->,thick] (C) --node[above]{$\rho_B$} (D);
          \draw[->,thick] (A) -- (C);
          \draw[->,thick] (B) --node[right]{${\alpha}$} (D);
          \draw[bend right,->,thick,label={[mylabel]{$\rho_{\tau}-\rho_T\cdot \tau_{R_B}$}},out=360-25,in=360-155] (T) to (C);
          \draw[bend left,->,thick,label={[mylabel]{$\rho_T \cdot \tau_A$}},out=25,in=155] (T) to (B);
          \draw[->,thick,dashed] (T) -- (A);
  \end{tikzpicture}
 \end{center}
\end{construction}
\begin{proof}[Correctness of the construction]
 $\pmatrow{\tau_{R_B}}{\tau_A}$ being a witness gives the equation
 \begin{equation}\label{equation:witness2}
  \tau_{R_B} \cdot \rho_B + \tau_A \cdot \alpha = \tau
 \end{equation}
 from which we can already see that if the constructed lift is well-defined as a morphism in $\mathcal{A}(\PC)$,
 then it really is a lift along our monomorphism.
 So we have to check that the morphism witness of our lift is correct.
 Multiplying $\eqref{equation:witness2}$ with $\rho_T$ from the left yields
 \begin{equation}
  \rho_T \cdot \tau_{R_B} \cdot \rho_B + \rho_T \cdot \tau_A \cdot \alpha = \rho_T \cdot \tau = \rho_{\tau} \cdot \rho_B
 \end{equation}
 and thus
 \begin{equation}
   (\rho_{\tau} - \rho_T \cdot \tau_{R_B}) \cdot \rho_B = (\rho_T \cdot \tau_A) \cdot \alpha
 \end{equation}
 which proves that the weak pullback induced morphism $\bmatrow{\rho_{\tau}-\rho_T \cdot \tau_{R_B}}{\rho_T \cdot \tau_A}$ is well-defined.
 Last, we compute
 \begin{align*}
  \rho_T \cdot \tau_A &= \bmatrow{\rho_{\tau}-\rho_T \cdot \tau_{R_B}}{\rho_T \cdot \tau_A} \cdot \secondproj \\
  &\stackrel{\eqref{equation:witness}}{=} \bmatrow{\rho_{\tau}-\rho_T \cdot \tau_{R_B}}{\rho_T \cdot \tau_A} \cdot \sigma \cdot \rho_A
 \end{align*}
 which shows that the morphism witness is correct.
\end{proof}

\subsection{Colifts along epimorphisms in Freyd categories}

Dually, we have to be able to construct colifts along
epimorphisms in $\mathcal{A}(\PC)$.

\begin{construction}[Colifts along epimorphisms]\label{construction:colift_along_epis}
 The cokernel projection of a given epimorphism
 \[
  \{\alpha, \rho_{\alpha}\}: (A \stackrel{\rho_A}{\longleftarrow} R_A) \longrightarrow (B \stackrel{\rho_B}{\longleftarrow} R_B)
 \]
 in $\mathcal{A}(\PC)$ is zero. 
 A witness of this fact 
 is given by a morphism
 $\pmatrow{\sigma_{R_B}}{\sigma_A}: B \rightarrow R_B \oplus A$ such that
 \begin{equation}\label{equation:w1}
  \sigma_{R_B} \cdot \rho_B + \sigma_A \cdot \alpha = \id_B
 \end{equation}
 (see Construction \ref{construction:cokernel}).
 Now, the following diagram shows us how to construct a colift along $\{\alpha, \rho_{\alpha}\}$
 for a given test morphism:
 \begin{center}
    \begin{tikzpicture}[label/.style={postaction={
          decorate,
          decoration={markings, mark=at position .5 with \node #1;}},
          mylabel/.style={thick, draw=black, align=center, minimum width=0.5cm, minimum height=0.5cm,fill=white}}]
          \coordinate (r) at (6,0);
          \coordinate (u) at (0,2);
          \node (A) {$\big(R_B \times_B A \longleftarrow R_K\big)$};
          \node (B) at ($(A)+0.8*(r)$) {$(A \stackrel{\rho_A}{\longleftarrow} R_A)$};
          \node (C) at ($(B) + 0.9*(r) + (u)$) {$(B \stackrel{\rho_B}{\longleftarrow} R_B)$};
          \node (T) at ($(B) + 0.9*(r) - (u)$) {$(T \stackrel{\rho_T}{\longleftarrow} R_T)$.};
          \draw[->,thick] (A) --node[above]{$\{\secondproj, \secondproj\}$} (B);
          \draw[->,thick] (B) --node[below,yshift=-0.1em,xshift=-1em]{$\{\tau, \rho_{\tau}\}$} (T);
          \draw[->,thick] (B) --node[above,xshift=-1.5em,yshift=0.5em]{$\{\alpha, \rho_{\alpha}\}$} (C);
          \draw[->,thick,dashed,label={[mylabel]{$\{\sigma_A \cdot \tau, \bmatrow{\id_{R_B} - \rho_B\cdot \sigma_{R_B}}{\rho_B \cdot \sigma_A} \cdot \sigma\}$}}] (C) -- (T);
          
          \node (Ae) at ($(A)+(-0.75,-0.15)$) {};
          \node (Te) at ($(T)-(-0.4,0.15)$) {};
          \draw[bend right,->,thick,dotted,label={[above]{$\sigma$}},out=360-25,in=360-155] (Ae) to (Te);
    \end{tikzpicture}
 \end{center}
 How to read this diagram: the solid horizontal arrow is the kernel embedding of 
 our epimorphism $\{\alpha, \rho_{\alpha}\}$ (see Construction \ref{construction:kernel}),
 where we set $R_K := (R_B \times_B A) \times_A R_A$.
 The dotted arrow is a witness for the composition of the kernel embedding with the test morphism $\{\tau, \rho_{\tau}\}$ being zero.
 The downwards pointing dashed arrow is the desired colift,
 whose morphism witness involves the weak pullback induced morphism given by the diagram
 \begin{center}
  \begin{tikzpicture}[label/.style={postaction={
          decorate,
          decoration={markings, mark=at position .5 with \node #1;}},
          mylabel/.style={thick, draw=none, align=center, minimum width=0.5cm, minimum height=0.5cm,fill=white}},
          baseline = (base)]
          \coordinate (r) at (2.5,0);
          \coordinate (u) at (0,2);
          \node (A) {$R_B \times_B A$};
          \node (B) at ($(A)+(r)$) {$A$};
          \node (C) at ($(A) - (u)$) {$R_B$};
          \node (D) at ($(B) - (u)$) {$B$.};
          \node (T) at ($(A) - (r) + 0.5*(u)$) {$R_B$};
          \node (base) at ($0.5*(A) + 0.5*(C)$) {};
          \draw[->,thick] (A) -- (B);
          \draw[->,thick] (C) --node[above]{$\rho_B$} (D);
          \draw[->,thick] (A) -- (C);
          \draw[->,thick] (B) --node[right]{${\alpha}$} (D);
          \draw[bend right,->,thick,label={[mylabel]{$\id_{R_B} - \rho_B \cdot \sigma_{R_B}$}},out=360-25,in=360-155] (T) to (C);
          \draw[bend left,->,thick,label={[mylabel]{$\rho_B \cdot \sigma_A$}},out=25,in=155] (T) to (B);
          \draw[->,thick,dashed] (T) -- (A);
  \end{tikzpicture}
 \end{center}
 
\end{construction}
\begin{proof}[Correctness of the construction]
 Multiplying \eqref{equation:w1} with $\rho_B$ from the left yields
 \begin{equation}
  \rho_B \cdot \sigma_{R_B} \cdot \rho_B + \rho_B \cdot \sigma_A \cdot \alpha = \rho_B
 \end{equation}
 and thus
 \begin{equation}
  (\rho_B \cdot \sigma_A) \cdot \alpha = ( \id_{R_B} - \rho_B \cdot \sigma_{R_B} ) \cdot \rho_B
 \end{equation}
 which shows that $\bmatrow{\id_{R_B} - \rho_B\cdot \sigma_{R_B}}{\rho_B \cdot \sigma_A}$ is well-defined.
 Well-definedness of the constructed colift as a morphism in $\mathcal{A}(\PC)$
 follows from the commutativity of the two inner squares in the diagram
 \begin{center}
  \begin{tikzpicture}[label/.style={postaction={
          decorate,
          decoration={markings, mark=at position .5 with \node #1;}},
          mylabel/.style={thick, draw=none, align=center, minimum width=0.5cm, minimum height=0.5cm,fill=white}},
          baseline = (base)]
          \coordinate (r) at (2.5,0);
          \coordinate (u) at (0,2);
          \node (A) {$R_B$};
          \node (B) at ($(A)+2.5*(r)$) {$R_B \times_B A$};
          \node (C) at ($(A) - (u)$) {$B$};
          \node (D) at ($(B) - (u)$) {$A$};
          \node (E) at ($(B) + (r)$) {$R_T$};
          \node (F) at ($(D) + (r)$) {$T$.};
          \node (base) at ($0.5*(A) + 0.5*(C)$) {};
          \draw[->,thick,dashed] (A) --node[above]{$\bmatrow{\id_{R_B} - \rho_B\cdot \sigma_{R_B}}{\rho_B \cdot \sigma_A}$} (B);
          \draw[->,thick] (A) --node[left]{$\rho_B$} (C);
          \draw[->,thick] (C) --node[above]{$\sigma_A$} (D);
          \draw[->,thick] (B) --node[left]{$\secondproj$} (D);
          \draw[->,thick] (B) --node[above]{$\sigma$} (E);
          \draw[->,thick] (D) --node[above]{$\tau$} (F);
          \draw[->,thick] (E) --node[right]{$\rho_T$} (F);
  \end{tikzpicture}
 \end{center}
It remains to show that the constructed colift really yields a colift for the given test morphism.
To this end, we multiply \eqref{equation:w1} with $\alpha$ from the left
\begin{equation}
 \alpha \cdot \sigma_{R_B} \cdot \rho_B + \alpha \cdot \sigma_A \cdot \alpha = \alpha
 \end{equation}
 and obtain
 \begin{equation}
 (\alpha \cdot \sigma_{R_B}) \cdot \rho_B = ( \id_A - \alpha \cdot \sigma_A) \cdot \alpha 
 \end{equation}
 which shows that $\bmatrow{\alpha \cdot \sigma_{R_B}}{\id_A - \alpha \cdot \sigma_A}: A \rightarrow R_B \times_B A$ is well-defined.
 Last, we compute
 \begin{align*}
 \bmatrow{\alpha \cdot \sigma_{R_B}}{\id_A - \alpha \cdot \sigma_A} \cdot \sigma \cdot \rho_{T}
 &=
  \bmatrow{\alpha \cdot \sigma_{R_B}}{\id_A - \alpha \cdot \sigma_A} \cdot \secondproj \cdot \tau \\   
  &=(\id_A - \alpha \cdot \sigma_A)\cdot \tau \\
  &={\tau - \alpha \cdot \sigma_A}\cdot \tau \\
 \end{align*}
 and see that we have found a witness for $\tau$ and $\alpha \cdot (\sigma_A\cdot \tau)$ being equal in $\mathcal{A}(\PC)$.
\end{proof}

\subsection{A constructive proof of the main theorem}

\begin{proof}[Proof of Theorem \ref{theorem:freyd}]
 Constructions \ref{construction:cokernel} shows that $\Freyd( \PC )$ always has cokernels.
 Moreover, if $\PC$ has weak kernels, then
 Constructions \ref{construction:kernel}, \ref{construction:lift_along_monos}, \ref{construction:colift_along_epis}
 show that $\mathcal{A}(\PC)$ is abelian.
 For the other direction, we construct weak kernels in $\PC$ using
 the operation for kernels in $\mathcal{A}(\PC)$:
 given a morphism $A \stackrel{\alpha}{\longrightarrow} B \in \PC$,
 we compute a kernel embedding
 \[\{\kappa, \rho_{\kappa}\}: (K \stackrel{\rho_K}{\longleftarrow} R_K) \longrightarrow (A \longleftarrow 0)\]
 of $\{\alpha,0\}:(A \longleftarrow 0) \longrightarrow (B \longleftarrow 0)$.
 Then $\kappa \cdot \alpha = 0 \in \PC$ since any witness for $\{\kappa, \rho_{\kappa}\}\cdot\{\alpha,0\}$
 being zero in $\mathcal{A}(\PC)$ factors over $0$.
 Thus, $\kappa: K \rightarrow A$ is a candidate for a weak kernel embedding of $\alpha$ in $\PC$.
 Let $T \stackrel{\tau}{\longrightarrow} A \in \PC$ such that $\tau \cdot \alpha = 0 \in \PC$.
 We use the universal property of the kernel in $\mathcal{A}(\PC)$
 to compute the dashed arrow in the following commutative diagram:
 \begin{center}
    \begin{tikzpicture}[label/.style={postaction={
          decorate,
          decoration={markings, mark=at position .5 with \node #1;}},
          mylabel/.style={thick, draw=black, align=center, minimum width=0.5cm, minimum height=0.5cm,fill=white}}]
          \coordinate (r) at (4.5,0);
          \coordinate (u) at (0,1);
          \node (A) {$(A \longleftarrow 0)$};
          \node (B) at ($(A)+0.8*(r)$) {$(B \longleftarrow 0)$.};
          \node (K) at ($(A) - 1.3*(r) + (u)$) {$(K \stackrel{\rho_K}{\longleftarrow}R_K)$};
          \node (T) at ($(A) - 1.3*(r) - (u)$) {$(T \longleftarrow 0)$};
          \draw[->,thick] (A) --node[above]{$\{\alpha,0\}$} (B);
          \draw[->,thick] (T) --node[below,yshift=-0.2em]{$\{\tau,0\}$} (A);
          \draw[->,thick] (K) --node[above,xshift=0em,yshift=0.3em]{$\{\kappa, \rho_{\kappa}\}$} (A);
          \draw[->,thick,dashed] (T) --node[left]{$\big\{u, \rho_u\big\}$} (K);
          
          \node (Be) at ($(B)-(-0.4,0.15)$) {};
          \node (Te) at ($(T)+(-0.75,-0.15)$) {};
          
    \end{tikzpicture}
 \end{center}
 Again, from $\{u, \rho_u\} \cdot \{\kappa, \rho_{\kappa}\} = \{\tau,0\} \in \mathcal{A}(\PC)$ we can deduce $u \cdot \kappa = \tau \in \PC$, thus, we really have constructed a weak kernel of $\alpha$.
\end{proof}

\begin{remark}
 For an actual computer implementation of an abelian category $\AC$,
 it is a useful feature to have \textbf{decidable equality} of morphisms (see, e.g., Remark \ref{remark:isomorphism}
 for an application in the context of Freyd categories).
 We call categories (in constructive contexts) with decidable equality of morphisms \textbf{computable}.
 In the case of Freyd categories,
 $\Freyd( \PC )$ is computable if and only if
 $\PC$ \textbf{has decidable lifts},
 which means (by our constructive interpretation)
 to have an algorithm that actually create lifts
 or to disproves\footnote{
 A good example for such a disproval is coming across the equation $0 = 1$ while performing the Gaussian elimination algorithm
 for solving a linear system over a field.
 } their existence).
\end{remark}

\begin{corollary}\label{corollary:freyd}
 Let $\PC$ be an additive category.
 Then $\mathcal{A}(\PC)$ is a computable abelian category
 if and only if $\PC$ has weak kernels and decidable lifts.
\end{corollary}

\begin{remark}
Note that if $\PC$ has weak kernels,
then (by our constructive interpretation) we already have an algorithm for lifting
\emph{some} cospans, namely those representing
a test situation for the weak kernel (see Definition \ref{definition:weak_kernels_3}).
But it is impossible to derive
from such an algorithm one for general lifts:
we will see an example of a computable additive category $\PC$
with weak kernels
but with a computationally undecidable lifting problem
in Subsection \ref{subsection:add_cat_dec_eq_kernels_undec_lifts}.
\end{remark}

\subsection{The induced functor}
In this subsection we single out
the most important constructive aspect
of the idea that the Freyd category
is a universal way to add cokernels to $\PC$.

\begin{construction}
Let $\PC$ be an additive category.
Given the data:
\begin{theoremenumerate}
 \item An additive category $\TC$.
 \item A functor $F: \PC \rightarrow \TC$.
 \item An \operation that \constructs for given $A \stackrel{\alpha}{\longrightarrow} B$ in $\PC$
       a cokernel object $\cokernel( F( \alpha ) )$
       (along with its cokernel projection and universal property) of $F(A) \stackrel{F(\alpha)}{\longrightarrow} F(B)$.\label{construction:induced_functor_3}
\end{theoremenumerate}
Then we can construct an \textbf{induced functor} $U: \Freyd( \PC ) \rightarrow \TC$ as follows.
\begin{itemize}
 \item An object $(A \stackrel{\rho_A}{\longleftarrow} R_A)$ in $\mathcal{A}(\PC)$ is mapped to $\cokernel( F( \rho_A ) )$.
 \item Given a morphism
 \[
  \{\alpha, \rho_{\alpha}\}: (A \stackrel{\rho_A}{\longleftarrow} R_A) \longrightarrow (B \stackrel{\rho_B}{\longleftarrow} R_B) \in \Freyd( \PC ),
 \]
 we map it to the morphism induced by the universal property of the cokernels:
 \begin{center}
         \begin{tikzpicture}[label/.style={postaction={
          decorate,
          decoration={markings, mark=at position .5 with \node #1;}},
          mylabel/.style={thick, draw=none, align=center, minimum width=0.5cm, minimum height=0.5cm,fill=white}}]
          \coordinate (r) at (3.5,0);
          \coordinate (u) at (0,2);
          \node (A) {$F(A)$};
          \node (B) at ($(A)+(r)$) {$F(R_A)$};
          \node (C) at ($(A) - (u)$) {$F(B)$};
          \node (D) at ($(B) - (u)$) {$F(R_B)$.};
          \node (CrhoA) at ($(A) - (r)$) {$\cokernel( F( \rho_A ) )$};
          \node (CrhoB) at ($(C) - (r)$) {$\cokernel( F( \rho_B ) )$};
          \draw[->,thick] (B) --node[above]{$F(\rho_A)$} (A);
          \draw[->,thick] (D) --node[above]{$F(\rho_B)$} (C);
          \draw[->,thick] (A) --node[left]{$F(\alpha)$} (C);
          \draw[->,thick] (A) -- (CrhoA);
          \draw[->,thick] (C) -- (CrhoB);
          \draw[->,thick] (B) --node[right]{$F(\rho_{\alpha})$} (D);
          \draw[->,thick,dashed] (CrhoA) -- (CrhoB);
          \end{tikzpicture}
        \end{center}
\end{itemize}
\end{construction}
\begin{proof}[Correctness of the construction]
 The morphism induced by the universal property of the cokernel is independent
 of the morphism witness. 
 Now, correctness follows from the functoriality of the cokernel.
\end{proof}

\begin{remark}
 If we have two \operations $\cokernel$, $\cokernel'$
 as instances of Construction \ref{construction:induced_functor_3},
 then the corresponding induced functors $U$, $U'$
 are naturally isomorphic.
\end{remark}

%

\section{Interpretations of Freyd categories}\label{section:applications_modules}

In this subsection we want to give several interpretations of $\Freyd( \PC )$ for specific inputs $\PC$
that are all of the same spirit: the objects in $\Freyd( \PC )$
represent finitely presented objects in some abelian category $\AC$.
The first theorem in this section
is the main tool for giving these interpretations.

\begin{theorem}\label{theorem:classical_equivalence}
 Let $\PC$ be an additive category,
 $\AC$ an abelian category, and $F: \PC \rightarrow \AC$ a full and faithful functor 
 such that $F(P)$ is a projective object for all $P \in \PC$.
 Then the induced functor
 \[
  U: \Freyd( \PC ) \rightarrow \AC
 \]
 is an equivalence 
 between
 $\Freyd( \PC )$
 and the full subcategory $\BC$ of $\AC$
 generated by those objects $A$
 for which there exist $P_A,Q_A \in \PC$
 and
 an exact sequence
 \[
  0 \longleftarrow A \stackrel{\epsilon_A}{\longleftarrow} F(P_A) \stackrel{\delta_A}{\longleftarrow} F(Q_A).
 \]
\end{theorem}
\begin{proof}
 We define an inverse functor $V: \BC \rightarrow \Freyd( \PC )$ as follows.
 We have an operation on objects sending $A \in \BC$ to $(P_A \stackrel{d_A}{\longleftarrow} Q_A) \in \Freyd( \AC )$,
 where $d_A$ is a preimage of $\delta_A$.
 
 Let $B \in \BC$ with $P_B, Q_B \in \PC$
 and $0 \longleftarrow B \stackrel{\epsilon_B}{\longleftarrow} F(P_B) \stackrel{\delta_B}{\longleftarrow} F(Q_B)$ exact.
 Using all of our assumptions on $F$, we can conclude:
 for $\alpha: A \rightarrow B \in \BC$,
 there exists a morphism $p:P_A \rightarrow P_B$ such that
 \begin{center}
         \begin{tikzpicture}[label/.style={postaction={
          decorate,
          decoration={markings, mark=at position .5 with \node #1;}},
          mylabel/.style={thick, draw=none, align=center, minimum width=0.5cm, minimum height=0.5cm,fill=white}}]
          \coordinate (r) at (3.5,0);
          \coordinate (u) at (0,2);
          \node (A) {$F(P_A)$};
          \node (B) at ($(A)+(r)$) {$F(Q_A)$};
          \node (C) at ($(A) - (u)$) {$F(P_B)$};
          \node (D) at ($(B) - (u)$) {$F(Q_B)$};
          \node (CrhoA) at ($(A) - (r)$) {$A$};
          \node (CrhoB) at ($(C) - (r)$) {$B$};
          \draw[->,thick] (B) --node[above]{$F(d_A)$} (A);
          \draw[->,thick] (D) --node[above]{$F(d_B)$} (C);
          \draw[->,thick] (A) --node[left]{$F(p)$} (C);
          \draw[->,thick] (A) --node[above]{$\epsilon_A$} (CrhoA);
          \draw[->,thick] (C) --node[above]{$\epsilon_B$} (CrhoB);
          \draw[->,thick] (B) --node[right]{$F(q)$} (D);
          \draw[->,thick] (CrhoA) --node[left]{$\alpha$} (CrhoB);
          \end{tikzpicture}
        \end{center}
 commutes for some morphism $q: Q_A \rightarrow Q_B$.
 Given another morphism $p': P_A \rightarrow P_B$ with this property,
 we have $(F(p) - F(p')) \cdot \epsilon_B = 0$.
 Again using all of our assumptions on $F$,
 we can conclude the existence of $\lambda: P_A \rightarrow Q_B$
 such that $\lambda \cdot d_B = p - p'$, which proves
 that $p$ and $p'$ are equal as morphisms
 from $(P_A \stackrel{d_A}{\longleftarrow} Q_A)$
 to $(P_B \stackrel{d_B}{\longleftarrow} Q_B)$ in $\Freyd( \PC )$.
 Thus, using the axiom of unique choice, $\alpha \mapsto p$ defines a well-defined action of $V$ on morphisms.
 Furthermore, $U$ and $V$ are readily seen to be mutual inverses.
\end{proof}

\subsection{Finitely presented modules}

\begin{example}[Interpretation as finitely presented modules]\label{example:fpmod}
 Let $R$ be a ring. We denote the abelian category of left $R$-modules
 by $R\Modl$. We define $\Rows_R$ to be the full subcategory of $R\Modl$ generated by 
 all row modules $R^{1 \times n}$ for $n \in \N_0$ considered as free left modules.
 Morphisms $\Hom_{\Rows_R}( R^{1 \times m}, R^{1 \times n} )$ can be naturally identified with matrices $R^{m \times n}$ for $m,n \in \N_0$.
 Then $\Rows_R \subseteq R\Modl$
 is a full and faithful embedding of projective objects.
 From Theorem \ref{theorem:classical_equivalence}
 we conclude
 \[
  \Freyd( \Rows_R ) \simeq R\fpmodl,
 \]
 where $R\fpmodl$ is the category of finitely presented left $R$-modules.
 
 The whole example also works for right modules by considering the full subcategory $\Cols_R \subseteq \Modr R$ of right column modules
 instead.
 Note that if $M \in R^{n \times m}$ represents a morphism from $R^{m \times  1}$ to $R^{n \times 1}$ in $\Cols_R$, 
 reinterpreting it as a morphism from $R^{1 \times n}$ to $R^{1 \times m}$ in $\Rows_R$
 yields an equivalence $\Cols_R \simeq \Rows_R^{\op}$. Furthermore,  $\Rows_R^{\op} \simeq \Rows_{R^{\op}}$. We conclude:
 \[
  \mathcal{A}( \Cols_R ) \simeq \mathcal{A}( \Rows_R^{\op} ) \simeq \mathcal{A}( \Rows_{R^{\op}} ) \simeq \fpmodr R,
 \]
 where $\fpmodr R$ denotes the category of finitely presented right $R$-modules.
\end{example}

In \cite{BL} computable rings are introduced. We give a definition
and a characterization of such rings using Freyd categories.

\begin{definition}
 A ring $R$ that
 \begin{theoremenumerate}
  \item is \textbf{left coherent}, i.e., 
        for a given matrix $A$ with coefficients in $R$
        we can compute a matrix $L$ such that $LA = 0$ and
        for all matrices $T$ such that $TA = 0$,
        there exists a matrix $U$ such that $UL = T$ (that means $L$ generates the row syzygies),\label{definition:computable_1}
  \item \textbf{has decidable lifts}, i.e., there is an algorithm to decide solvability and to construct a particular solution
        of a linear systems $XA = B$
        for given matrices $A$, $B$ with coefficients in $R$,
 \end{theoremenumerate}
 is called \textbf{left computable}. A ring $R$ is \textbf{right computable} if $R^{\op}$
 is left computable.
 If $R$ is left and right computable we simply call it \textbf{computable}.
\end{definition}

\begin{remark}
 Be aware of the existential quantifiers
 in Definition \ref{definition:computable_1}.
 By our constructive interpretation we regard left coherent rings
 as being equipped with an algorithm for 
 computing $L$ (for given $A$) \emph{as well as} $U$ (for given $A$ and $T$).
\end{remark}

\begin{remark}
 A left computable ring $R$ has decidable equality,
 since $a = b$ if and only if $x\cdot 0 = (a-b)$ is solvable
 for $a,b \in R$ (and similar for right computable rings).
\end{remark}

\begin{theorem}\label{theorem:rings_and_freyd}
 Let $R$ be a ring.
 \begin{theoremenumerate}
  \item $R$ is left coherent if and only if $\Freyd( \Rows_R )$ is abelian.\label{theorem:rings_and_freyd_1}
  \item $R$ is left computable if and only if $\Freyd( \Rows_R )$ is computable abelian.\label{theorem:rings_and_freyd_2}
 \end{theoremenumerate}
\end{theorem}
\begin{proof}
 $R$ being left coherent is equivalent to $\Rows_R$ having weak kernels.
 So, the first claim follows from Theorem \ref{theorem:freyd}.
 $R$ having decidable lifts is equivalent to $\Rows_R$ having decidable lifts,
 and the second claim follows from Corollary \ref{corollary:freyd}
\end{proof}

\subsection{Finitely presented graded modules}

\begin{example}[Interpretation as finitely presented graded modules]\label{example:fpgrmod}
 Let $G$ be a group and $S$
 a $G$-graded ring, i.e., a ring with
 a direct sum decomposition  $S = \bigoplus_{d \in G} S_d$ into 
 abelian groups such that $S_{d} \cdot S_{e} \subseteq S_{d\cdot e}$ for all $d,e \in G$.
 A graded left (resp. right) module is given by a left (resp. right) $S$-module
 $M$ equipped with a direct sum decomposition $M = \bigoplus_{d \in G} M_d$
 into abelian groups such that $S_d \cdot M_e \subseteq M_{d \cdot e}$ (resp. $M_e \cdot S_d \subseteq M_{e \cdot d})$
 for all $d,e \in G$. Graded left (resp. right) $S$-module homomorphisms are 
 given by left (resp. right) $S$-module homomorphisms respecting the grading.
 We denote the corresponding abelian category by $S\grModl$ (resp. $\grModr S$).
 
 For a graded left (resp. right) module $M$ and given $e \in G$, we denote by $M(e)$ the $e$-th shift of $M$,
 i.e., the graded left module with $M(e)_d := M_{d\cdot e}$ (resp. the graded right module with $M(e)_d := M_{e\cdot d}$).
 The full subcategory generated by graded left (resp. right) modules of the form $S(e_1) \oplus \dots \oplus S(e_r)$
 for $e_1, \dots, e_r \in G$ is denoted by $\grRows_S$ (resp. $\grCols_S$).
 Morphisms in $\grRows_S$ (resp. $\grCols_S$) from $S(d_1) \oplus \dots \oplus S(d_r)$ to $S(e_1) \oplus \dots \oplus S(e_s)$
 for $d_1, \dots, d_r, e_1, \dots, e_s \in G$ can be naturally identified with
 matrices $H \in S^{r \times s}$ (resp. $H \in S^{s \times r}$) with homogeneous entries 
 $H_{ij} \in S_{d_i^{-1} \cdot e_j}$ 
 (resp. $H_{ji} \in S_{e_j \cdot d_i^{-1}}$)
 for $i=1, \dots,r, j = 1,\dots s$.
 Since $\grRows_S \subseteq S\grModl$ (resp. $\grCols_S \subseteq \grModr S$)
 are full and faithful embeddings of projective objects,
 we conclude (using Theorem \ref{theorem:classical_equivalence})
 \[
  \Freyd(\grRows_S) \simeq S\fpgrmodl \text{~(resp.~$ \Freyd(\grCols_S) \simeq \fpgrmodr S$)}
 \]
 where $S\fpgrmodl$ (resp. $\fpgrmodr$ S) is the category of finitely presented graded left (resp. right) $S$-modules.
\end{example}

For an implementation of $\Freyd(\grRows_S)$
as a computable abelian category,
we need $\grRows_S$ to have weak kernels and decidable lifts (Corollary \ref{corollary:freyd}).
These requirements for $\grRows_S$ concisely encode the following specifications needed in an actual implementation:
 we need
 data structures for elements in $G$, a constructor for the neutral element $e_G \in G$,
 algorithms for multiplication, inversion, and equality in $G$.
 Furthermore, we need
 data structures for elements in $S_d$ ($d \in G$),
 constructors for $1 \in S_{e_G}$ and $0 \in S_d$ ($d \in G$),
 algorithms $S_d \times S_{d'} \rightarrow S_{dd'}$ for multiplication ($d,d' \in G$), 
 algorithms $S_d \times S_{d} \rightarrow S_{d}$ for addition and subtraction ($d \in G$),
 an algorithm for equality in $S_d$ $(d \in G)$.
 Furthermore, for weak kernels and decidable lifts in $\grRows_S$,
 we need an algorithm for computing \emph{homogeneous} row syzygies of a matrix with \emph{homogeneous} entries,
 and an algorithm for deciding the existence and in the affirmative case
 computing a solution of a linear system $X A = B$,
 where $A,B$ and the solution $X$ 
 are matrices with \emph{homogeneous} entries.

\subsection{Finitely presented functors}

We give a general interpretation of the Freyd category
in terms of finitely presented functors (cf.~\cite[Corollary 3.9]{BelFredCats}).

\begin{example}[Classical interpretation as finitely presented functors]\label{example:fpfunctors}
 Given an additive category $\PC$,
 let
 \[
  Y: \PC \longrightarrow \Hom( \PC^{\op}, \Ab ): P \mapsto (-,P)
 \]
 be the Yoneda embedding.
 Here, $\Hom( \PC^{\op}, \Ab )$
 denotes the abelian category of contravariant
 functors from $\PC$ to the category of abelian groups $\Ab$,
 and $(-,P)$ denotes the contravariant
 $\Hom$ functor.
 By Yoneda's Lemma,
 $Y$ is full and faithful.
 Again following from Yoneda's Lemma, $(-,P)$
 is a projective object.
 We conclude (using Theorem \ref{theorem:classical_equivalence})
 \[
  \Freyd(\PC) \simeq \fp( \PC^{\op}, \Ab ),
 \]
 where $\fp( \PC^{\op}, \Ab )$ is the category of \textbf{contravariant finitely presented functors},
 i.e., the objects are functors $F: \PC^{\op} \rightarrow \Ab$ for
 which there exists $A,B \in \PC$ and an exact sequence of functors
 \begin{center}
    \begin{tikzpicture}[label/.style={postaction={
    decorate,
    decoration={markings, mark=at position .5 with \node #1;}},
    mylabel/.style={thick, draw=none, align=center, minimum width=0.5cm, minimum height=0.5cm,fill=white}}]
    \coordinate (r) at (2.5,0);
    \coordinate (u) at (0,2);
    \node (D) {$0$};
    \node (C) at ($(D) + 0.75*(r)$) {$F$};
    \node (B) at ($(C)+(r)$) {$(-,A)$};
    \node (A) at ($(B)+(r)$) {$(-,B)$,};
    
    \draw[->,thick] (A) -- (B);
    \draw[->,thick] (B) -- (C);
    \draw[->,thick] (C) -- (D);
    \end{tikzpicture}
\end{center}
and morphisms are given by natural transformations.
 Similarly, we get
 \[
  \Freyd(\PC^{\op}) \simeq \fp( \PC, \Ab ),
 \]
 where $\fp( \PC, \Ab )$ is the category of \textbf{covariant finitely presented functors},
 i.e., functors $F: \PC \rightarrow \Ab$ for
 which there exists $A,B \in \PC$ and an exact sequence of functors
 \begin{center}
    \begin{tikzpicture}[label/.style={postaction={
    decorate,
    decoration={markings, mark=at position .5 with \node #1;}},
    mylabel/.style={thick, draw=none, align=center, minimum width=0.5cm, minimum height=0.5cm,fill=white}}]
    \coordinate (r) at (2.5,0);
    \coordinate (u) at (0,2);
    \node (D) {$0$};
    \node (C) at ($(D) + 0.75*(r)$) {$F$};
    \node (B) at ($(C)+(r)$) {$(A,-)$};
    \node (A) at ($(B)+(r)$) {$(B,-)$.};
    
    \draw[->,thick] (A) -- (B);
    \draw[->,thick] (B) -- (C);
    \draw[->,thick] (C) -- (D);
    \end{tikzpicture}
\end{center}

\end{example}

We will study
finitely presented functors
in the case $\PC$ abelian
in Section \ref{section:applications_fpfunctors}.

\section{Computationally undecidable lifting and colifting problems}\label{section:undecidable}

In this section we provide several examples
concerning computationally undecidable problems.
\begin{enumerate}
 \item We give an example of
a ring $R$
with decidable equality\footnote{such rings are also called \textbf{discrete} in \cite{MRRConstructiveAlgebra}}
whose lifting and colifting problems are computationally undecidable.
\item From such an $R$,
we build a ring $P$ with decidable equality whose colifting problem is still computationally undecidable,
but now, $P$ has decidable lifts.
\item We use $P$ for the construction of an
additive category $\PC$ with decidable equality, having weak kernels,
but whose lifting problem is computationally undecidable.
\end{enumerate}

But first, we will explain what we mean by computationally undecidable problems.

\begin{definition}
 Let $\PC$ be an additive category.
 We say the \textbf{lifting problem for $\PC$ is computationally undecidable}
 if $\PC$ having decidable lifts would imply the decidability 
 of a problem that is known to be undecidable
 by Turing machines.
 We proceed analogously for colifts.
\end{definition}

\begin{definition}
 Let $R$ be a ring.
 We say the \textbf{lifting problem for $R$ is computationally undecidable}
 if the lifting problem for $\Rows_R$ is computationally undecidable.
 We can rephrase this condition in terms of equations for $R$:
 having an algorithm for deciding and finding particular solutions of left-sided equations $XA = B$
 would imply the decidability 
 of a problem that is known to be undecidable
 by Turing machines.
 We proceed analogously for colifts.
\end{definition}

\subsection{A ring with decidable equality and computationally undecidable lifting and colifting problem}\label{subsection:computable_lift_computable_colift}
We describe the famous word problem for finitely presented groups:
given a finite set $\Xcal$, let $\Fr(\Xcal)$ denote the free group over $\Xcal$.
Let furthermore $\Rcal \subseteq \Fr(\Xcal)$ be a finite subset and
$N$ the normal subgroup generated by $\Rcal$.
The word problem for $H := \Fr(\Xcal)/N$ (more precisely for $\Xcal$ and $\Rcal$) is the algorithmic problem of deciding
whether a given $w \in \Fr(\Xcal)$ represents the identity element in $H$,
i.e., whether there is an algorithm rendering the (classically trivial\footnote{by the law of excluded middle}) proposition
\[
 \forall w \in \Fr(\Xcal): (w \in N) \vee (w \not\in N)
\]
constructive. 
There are known concrete instances for $\Xcal$ and $\Rcal$
for which the word problem is undecidable \cite{BooneWord}
when we use Turing machines as a model of computation.
Let $\Xcal,\Rcal$ be such an instance.

Let $G := \Fr(\Xcal) \times \Fr(\Xcal)$ and let
$M \subseteq G$
denote the equivalence relation induced on $\Fr(\Xcal)$ by $N$,
i.e., the set of all pairs $(w_1, w_2)$ such that $w_1$ and $w_2$ represent the same element in $H$.
Then $M$ is the so-called Mihailova subgroup of $G$ (introduced in \cite{Miha58}),
and deciding whether $w \in \Fr(\Xcal)$ represents the identity in $H$
is equivalent to deciding $(w,1) \in M$.
It is easy to see that $M$ is finitely generated as a subgroup by the finite set
\[
 M' := \{ (x,x) \mid x \in \Xcal \} \cup \{ (1,r) \mid r \in \Rcal \}.
\]
We set $R := k[G]$, i.e., $R$ is the group ring of $G$ with coefficients in a field with decidable equality\footnote{
The Gaussian algorithm turns
such a field
into a
left and right computable ring.
} $k$.
Then $R$ is a $k$-algebra with decidable equality since the word problem in $G$ is decidable.
Furthermore, we claim that the lifting and colifting problems for $R$ are computationally undecidable.
This follows from the following lemma.

\begin{lemma}\label{lemma:kgideal}
 Let $k$ be a field, $G$ be a group, and $M \subseteq G$ be a subgroup.
 We define the right ideal
 \[
  I := \langle 1 - m \mid m \in M \rangle_{k[G]}
 \]
 of the group ring $k[G]$. Given $g \in G$, we have
 \[
  g \in M ~\Longleftrightarrow~ (1 - g) \in I.
 \]
 Furthermore, if $M$ is generated by the subset $M' \subseteq M$,
 then for the right ideal $I_{M'} := \langle 1 - m' \mid m' \in M' \rangle_{k[G]}$, we have
 \[
  I = I_{M'}.
 \]
\end{lemma}
\begin{proof}
 Let $\epsilon: G \rightarrow M\backslash G$ be the canonical $G$-equivariant map from $G$ into the set of right cosets of $M$ in $G$.
 Then $\epsilon$ induces a map of right $k[G]$-modules
 \[
  k[\epsilon]: k[G] \rightarrow k[M\backslash G]
 \]
 with $I \subseteq \kernel(k[\epsilon])$. 
 If $(1-g) \in I$, then $k[\epsilon](1-g) = 0 = M - Mg$,
 which gives $M = Mg$ and thus $g \in M$.
 
 For the second claim, let $m_1,m_2 \in M$ such that $(1-m_1), (1-m_2) \in I_{M'}$.
 Then 
 \[
  (1 - m_1)(-m_1^{-1}) = 1 - m_1^{-1}\in I_{M'}
 \]
 and 
 \[
 (1-m_1)m_2 + (1-m_2) = 1 - m_1 m_2 \in I_{M'}.
 \]
\end{proof}

\begin{remark}
 The choice of formulating Lemma \ref{lemma:kgideal} 
 in terms of right ideals was arbitrary.
 It is also valid (with an analogous proof)
 in terms of left ideals.
\end{remark}

\begin{corollary}\label{corollary:R_undecidable_colifts}
 The lifting and colifting problems are computationally undecidable for $k[\Fr(\Xcal) \times \Fr(\Xcal)]$.
\end{corollary}
\begin{proof}
 For $g \in \Fr(\Xcal) \times \Fr(\Xcal)$
 and $M'$ the finite generating set of the Mihailova subgroup,
 deciding whether $1-g$ lies in the finitely generated right ideal $I_{M'}$
is equivalent to deciding whether there exists a colift
of the diagram
\begin{center}
        \begin{tikzpicture}[label/.style={postaction={
          decorate,
          decoration={markings, mark=at position .5 with \node #1;}},
          mylabel/.style={thick, draw=none, align=center, minimum width=0.5cm, minimum height=0.5cm,fill=white}}]
          \coordinate (r) at (3.5,0);
          \coordinate (u) at (0,2);
          \node (m) {$R^{1 \times 1}$};
          \node (n) at ($(m)+(r)$) {$R^{1 \times |M'|}$.};
          \node (r1) at ($(m) + (u)$) {$R^{1 \times 1}$};
          \draw[->,thick] (m) --node[below]{$\left(1-m'\right)_{m' \in M'}$} (n);
          \draw[->,thick,dashed] (n) -- (r1);
          \draw[->,thick] (m) --node[left]{$(1 - g)$} (r1);
    \end{tikzpicture}
  \end{center}
Analogously, checking whether an element lies in a finitely generated left ideal 
can be formulated in terms of a lifting problem.
\end{proof}

\subsection{A ring with decidable equality, decidable lifts, and a computationally undecidable colifting problem}\label{subsection:computable_lift_undecidable_colift}

For simplifying our exposition, 
we axiomatize those properties of $R$
that are needed for our construction in this subsection:
let $R$ be a $k$-algebra with decidable equality
such that
\begin{enumerate}
 \item $R$ has an enumerable $k$-basis $v_1, v_2, v_3, \dots$,
 \item the colifting problem for $R$ is computationally undecidable.
\end{enumerate}
The group ring $k[\Fr(\Xcal) \times \Fr(\Xcal)]$
satisfies these requirements since $(1)$ the elements of $\Fr(\Xcal) \times \Fr(\Xcal)$
form a $k$-basis that can be enumerated
and $(2)$
the colifting problem for $k[\Fr(\Xcal) \times \Fr(\Xcal)]$ is computationally undecidable
due to Corollary \ref{corollary:R_undecidable_colifts}.

\textbf{Goal of this subsection:} the creation of a ring $P$ with decidable equality,
decidable lifts, and a computationally undecidable colifting problem.

\textbf{The main idea:}
we create $P$ from $R$ by ``adding'' operators that
help solving left-sided equations 
\begin{equation}\label{equation:xac}
X \cdot A = B 
\end{equation}
for matrices $A,B$ over $R$, but that are of no use for right-sided equations $A \cdot X = B$.

\textbf{Helpful operators:}
as a $k$-vector space, we have
$R = R^i \oplus \overline{R^i}$, where we set for $i \in \N$
\begin{itemize}
 \itemsep0.5em 
 \item $R^i := \langle v_1, \dots, v_i \rangle_k$,
 \item $\overline{R^i} := \langle v_j \mid j > i \rangle_k$.
\end{itemize}
Given matrices $A,B$ over $R$, there exists a $d \in \N$ such 
that all entries of $A, B$ lie in $R^d$.
If $(x_{mn})_{mn}$ is a solution over $R$ for the left-sided equation \eqref{equation:xac},
then each entry $x_{mn}$ could be replaced 
by a $k$-linear operator $\omega_{mn}$ in 
\[
\Omega^d := \End_k( R^d ) \simeq k^{d \times d},
\]
where $\omega_{mn}$ acts like $x_{mn}$ on $R^d$ and projects the result back into $R^d$.
The crucial observation is the following:
replacing the matrix $(x_{mn})_{mn}$ by $( \omega_{mn} )_{mn}$ still gives a solution for the left-sided equation \eqref{equation:xac}.

Now, finding a solution of \eqref{equation:xac} with entries in $\Omega^d$ (where $A$ and $B$ are still defined over $R$)
can be done by making an ansatz and using linear algebra.
If no solution with entries in $\Omega^d$ exists, then there is also no solution with entries in $R$.
This motivates to ``add'' the operator spaces $\Omega^i$ for all $i \in \N$
to $R$, but in an ``unbalanced way'', i.e.,
the $\Omega^i$ must not be helpful for solving right-sided equations.

\textbf{The details:}
we extend the action of an operator $\omega \in \Omega^i$
from $R^i$ to $R$ by setting $\omega( \overline{R^i} ) = 0$.

\begin{construction}
 We construct a category $\CC$ enriched over $k$ as follows:
\begin{enumerate}
 \item $\CC$ consists of three objects that we denote by $z_1, z_2, z_3$.
 \item The homomorphism $k$-vector spaces in $\CC$ are determined by the 
 following diagram:
 \begin{center}
    \begin{tikzpicture}[label/.style={postaction={
          decorate,
          decoration={markings, mark=at position .5 with \node #1;}},
          mylabel/.style={thick, draw=black, align=center, minimum width=0.5cm, minimum height=0.5cm,fill=white}}]
          \coordinate (r) at (4.5,0);
          \coordinate (u) at (0,1);
          \node (A) {$z_1$};
          \node (B) at ($(A) + (r)$) {$z_2$};
          \node (C) at ($(B) + (r)$) {$z_3$};
          \draw[->,thick] (A) --node[above]{$R \oplus \left(\bigoplus_{i \in \N} \Omega^i\right)$} (B);
          \draw[->,thick] (B) --node[above]{$R$} (C);
    \end{tikzpicture}
 \end{center}
 This means, $\Hom_{\CC}(z_1, z_2) = R \oplus \left(\bigoplus_{i \in \N} \Omega^i\right)$,
 $\Hom_{\CC}(z_2, z_3) = R$, $\Hom(z_j, z_j) = k$ for $j = 1,2,3$,
 and $\Hom_{\CC}(z_1,z_3) = R$. The other $\Hom$-sets are given by $0$.
 Composition of two consecutive arrows $z_1 \rightarrow z_2$ and $z_2 \rightarrow z_3$
 is induced by the operations of $R$ and $\Omega^i$ on $R$.
\end{enumerate}
\end{construction}

Next, let $P$ denote the (path) algebra of $\CC$, i.e., its underlying vector space is
\begin{equation}\label{equation:definition_KCC}
P = \bigoplus_{i=1}^3\Hom_{\CC}(z_i,z_i) \oplus \Hom_{\CC}(z_1,z_2) \oplus \Hom_{\CC}(z_2,z_3) \oplus \Hom_{\CC}(z_1,z_3)  
\end{equation}
with multiplication given by composition in $\CC$ and
setting the multiplication of two non-composable morphisms to $0$.
Then $P$ has decidable equality.
%
%
Given an element $x \in P$,
we will write 
\[
 x = x^{(1,1)} + x^{(2,2)} + x^{(3,3)} + x^{(1,2)} + x^{(2,3)} + x^{(1,3)}
\]
for its decomposition w.r.t.\ the direct sum decomposition \eqref{equation:definition_KCC}.

\begin{lemma}\label{lemma:kCcolifts}
 The colifting problem is computationally undecidable for $P$.
\end{lemma}
\begin{proof}
 We reduce the computationally undecidable colifting problem of $R$
 to the colifting problem of $P$.
 Given matrices $A,B$ over $R$, 
 we interpret the entries of $A$ as
 elements in $\Hom_{\CC}( z_1, z_2 )$ and the entries in 
 $B$ as elements in $\Hom_{\CC}(z_1,z_3)$.
 Then $A \cdot X = B$ has a solution over $R$
 if and only if it has a solution over $P$, since
 a solution over $R$ can be interpreted as
 a solution with entries in $\Hom_{\CC}(z_2,z_3) \subseteq P$,
 and a solution $X = (x_{mn})_{mn}$ over $P$
 gives rise to a solution $(x_{mn}^{(2,3)})_{mn}$
 with entries in $\Hom_{\CC}(z_2,z_3) = R$.
\end{proof}

For the proof of the next lemma, we
introduce the following spaces for all $i \in \N$:
\begin{itemize}
\itemsep0.5em
 \item $\Hom_{\CC}^i(z_1,z_2) := R^i \oplus (\bigoplus_{j \leq i} \Omega^j) \subseteq \Hom_{\CC}(z_1,z_2)$ 
 \item $\overline{\Hom_{\CC}^i}(z_1,z_2) := \overline{R^i} \oplus (\bigoplus_{j > i} \Omega^j) \subseteq \Hom_{\CC}(z_1,z_2)$
 \item $\Hom_{\CC}^i(z_m,z_n) := R^i \subseteq \Hom_{\CC}(z_m,z_n)$ for $(m,n) \in \{(1,3),(2,3)\}$
 \item $\overline{\Hom_{\CC}^i}(z_m,z_n) := \overline{R^i} \subseteq \Hom_{\CC}(z_m,z_n)$ for $(m,n) \in \{(1,3),(2,3)\}$
 \item $P^i := \bigoplus_{i=1}^3\Hom_{\CC}(z_i,z_i) \oplus \Hom_{\CC}^i(z_1,z_2) \oplus \Hom_{\CC}^i(z_2,z_3) \oplus \Hom_{\CC}^i(z_1,z_3)$
 \item $\overline{P^i} := \overline{\Hom_{\CC}^i}(z_1,z_2) \oplus \overline{\Hom_{\CC}^i}(z_2,z_3) \oplus \overline{\Hom_{\CC}^i}(z_1,z_3)$
\end{itemize}
Given an element $x \in P$,
we will write 
\[
x = x^i + \overline{x^i} 
\]
for its decomposition in $P^i \oplus \overline{P^i}$.

\begin{theorem}\label{theorem:kC_lifts}
 $P$ has decidable lifts.
\end{theorem}
\begin{proof}
 Given matrices $A,B$ over $P$, we need to decide whether there
 exists a matrix $X$ over $P$ such that $X \cdot A = B$,
 and in the affirmative case construct such a matrix.
 There exist $d, d' \in \N$ such that all entries of $A$ and $B$ lie in $P^d$
 and such that such that $R^d \cdot R^d \subseteq R^{d'}$
 (which is actually equivalent to $P^d \cdot P^d \subseteq P^{d'}$).
 Our main claim is the following: 
 \begin{center}
 \fbox{
 $\exists X$ over $P$
 solving $X \cdot A = B$ $~\Leftrightarrow~$ $\exists X$ with all entries in $P^{d'}$ solving $X \cdot A = B$.
 }
 \end{center}
 Once we know that our main claim is true, $X \cdot A = B$ can simply be solved
 by making an ansatz with all entries of $X$ lying in $P^{d'}$.
 This ansatz yields a linear system over $k$ and can be dealt with since $k$ is a field with decidable equality.
 
 The ``only if'' direction of our main claim is the hard part, so let us assume that we are given a solution
 $X$ over $P$. 
 We define an operator $\sigma: P \rightarrow P^{d'}$
 with the idea that entrywise applied to $X$ it will also yield a solution:
 \[
  \sigma(x) := x^d + \tau( \overline{x^d} ),
 \]
 where we define
  \begin{align*}
  \tau: \overline{P^d} &\rightarrow P^{d'}: \\
  \overline{\Hom_{\CC}^d(z_1,z_2)} \ni p &\mapsto \left(\tau(p):R^{d'} \rightarrow R^{d'}: y \mapsto (p \cdot y)^{d'}\right)\\
  \overline{\Hom_{\CC}^d(z_2,z_3)} \ni q &\mapsto 0 \\
  \overline{\Hom_{\CC}^d(z_1,z_3)} \ni r &\mapsto r^{d'}
 \end{align*}
 Here, $\tau(p)$ denotes the operator in $\Omega^{d'} \subseteq \Hom_{\CC}(z_1,z_2)$
 that maps $y \in R^{d'} \subseteq \Hom_{\CC}(z_2,z_3)$ to $(p \cdot y)^{d'} \in R^{d'} \subseteq \Hom_{\CC}(z_1,z_3)$.
 
 We show that $\sigma$ entrywise applied to $X$ yields a solution:
 given $x_1, \dots, x_n \in P$ and $a_1, \dots, a_n \in P^d$ for $n \in \N_0$ such that
 $\sum_{i=1}^nx_i \cdot a_i \in P^d$, we claim that
 \[
  \sum_{i=1}^nx_i \cdot a_i = \sum_{i=1}^n \sigma(x_i) \cdot a_i.
 \]
 For this, it suffices to show that
 \begin{equation}\label{equation:claim1}
  \sum_{i=1}^n \overline{x_i^d} \cdot a_i = \sum_{i=1}^n \tau(\overline{x_i^d}) \cdot a_i.
 \end{equation}
 Due to our choice of $d'$, the term $\sum_{i=1}^n \overline{x_i^d} \cdot a_i = \sum_{i=1}^n x_i \cdot a_i - \sum_{i=1}^n x_i^d \cdot a_i$ lies in
 \begin{equation}\label{equation:subset}
  P^d \cdot P^d \subseteq \bigoplus_{i=1}^3\Hom_{\CC}(z_i,z_i) \oplus \Hom_{\CC}^d(z_1,z_2) \oplus \Hom_{\CC}^d(z_2,z_3) \oplus \Hom_{\CC}^{d'}(z_1,z_3).
 \end{equation}
 Now, to simplify our notation, we write
 \[
 \overline{x_i^d} = p_i + q_i + r_i \in \overline{\Hom_{\CC}^d}(z_1,z_2) \oplus \overline{\Hom_{\CC}^d}(z_2,z_3) \oplus \overline{\Hom_{\CC}^d}(z_1,z_3)
 \]
 for the decomposition of $\overline{x_i^d}$ w.r.t.\ the above direct sum.
 Rewriting the claim \ref{equation:claim1} gives 
 \begin{equation}\label{equation:claim2}
  \sum_{i=1}^n (p_i + q_i + r_i) \cdot a_i = \sum_{i=1}^n (\tau(p_i) + r_i^{d'} ) \cdot a_i.
 \end{equation}
 We have
 \[
  \sum_{i=1}^n q_i \cdot a_i = 0,
 \]
 since
 \[
  \sum_{i=1}^n q_i \cdot a_i = \sum_{i=1}^n q_i \cdot a_i^{(3,3)}
 \]
 both lies in $\overline{\Hom_{\CC}(z_2,z_3)^d}$ (since the $a_i^{(3,3)}$ act like scalars) and $\Hom_{\CC}(z_2,z_3)^d$ (due to \eqref{equation:subset}).
 Thus, the claim \eqref{equation:claim2} simplifies to
 \[
  \sum_{i=1}^n (p_i + r_i) \cdot a_i = \sum_{i=1}^n (\tau(p_i) + r_i^{d'} ) \cdot a_i,
 \]
 which is equivalent to
 \[
  \sum_{i=1}^n (p_i + r_i) \cdot (a_i^{(2,2)} + a_i^{(2,3)} + a_i^{(3,3)}) = \sum_{i=1}^n (\tau(p_i) + r_i^{d'} ) \cdot (a_i^{(2,2)} + a_i^{(2,3)} + a_i^{(3,3)}),
 \]
 which in turn is equivalent to the two equations
 \begin{equation}\label{equation:claim3}
  \sum_{i=1}^n p_i\cdot a_i^{(2,2)} = \sum_{i=1}^n \tau(p_i)\cdot a_i^{(2,2)}
 \end{equation}
 and
 \begin{equation}\label{equation:claim4}
  \sum_{i=1}^n p_i\cdot a_i^{(2,3)} + r_i\cdot a_i^{(3,3)} = \sum_{i=1}^n \tau(p_i)\cdot a_i^{(2,3)} + r_i^{d'}\cdot a_i^{(3,3)}.
 \end{equation}
 First, we deal with \eqref{equation:claim3}:
 \[
  \sum_{i=1}^n p_i\cdot a_i^{(2,2)} = 0
 \]
 since this term both lies in 
 $\overline{\Hom_{\CC}(z_1,z_2)^d}$ (since the $a_i^{(2,2)}$ act like scalars) and $\Hom_{\CC}(z_1,z_2)^d$ (due to \eqref{equation:subset}).
 Because $\tau$ is a linear map, we also have
 \[
  \sum_{i=1}^n \tau(p_i)\cdot a_i^{(2,2)} = 0.
 \]
 Finally, we deal with \eqref{equation:claim4}:
 \begin{align*}
  \sum_{i=1}^n p_i\cdot a_i^{(2,3)} + r_i\cdot a_i^{(3,3)} &= (\sum_{i=1}^n p_i\cdot a_i^{(2,3)} + r_i\cdot a_i^{(3,3)})^{d'} & \text{(due to \eqref{equation:subset})} \\
                                                           &= \sum_{i=1}^n (p_i\cdot a_i^{(2,3)})^{d'} + r_i^{d'}\cdot a_i^{(3,3)}& \text{(linearity)} \\
                                                           &= \sum_{i=1}^n \tau(p_i)\cdot a_i^{(2,3)} + r_i^{d'}\cdot a_i^{(3,3)}& \text{(definition of $\tau$).} \\
 \end{align*}

\end{proof}

\subsection{An additive category with decidable equality, having kernels, and a computationally undecidable lifting problem}\label{subsection:add_cat_dec_eq_kernels_undec_lifts}
Let $P$ be the ring constructed in Subsection \ref{subsection:computable_lift_undecidable_colift}.
It has decidable lifts but a computationally undecidable colifting problem.
We set $\PC := \mathcal{A}( \Rows_{P} )^{\op}$.
Then $\PC$ is additive 
with decidable equality,
since $P$ has decidable lifts.
Furthermore it has weak kernels (even kernels),
since $\mathcal{A}( \Rows_{P} )$
has cokernels by Construction \ref{construction:cokernel}.

However, the lifting problem for $\PC$ is computationally undecidable:
an algorithm for deciding lifts in $\PC = \mathcal{A}( \Rows_{P} )^{\op}$,
i.e., colifts in $\mathcal{A}( \Rows_{P} )$,
immediately gives an algorithm for deciding colifts in $\Rows_{P}$
(because $\Rows_{P} \subseteq \Freyd( \Rows_{P} )$ is a full embedding).
But the colifting problem is computationally undecidable for $P$ by Lemma \ref{lemma:kCcolifts}.
We have proven the main theorem of this section:

\begin{theorem}\label{theorem:add_cat_dec_eq_kernels_undec_lifts}
The category $\Freyd(\Rows_{P})^{\op}$ is additive with decidable equality,
has kernels, but
its lifting problem is computationally undecidable.
\end{theorem}

\begin{remark}
 Theorem \ref{theorem:add_cat_dec_eq_kernels_undec_lifts}
 shows that even if we have an additive category
 with decidable equality and weak kernels
 (this includes having an algorithm for lifting those cospans
 representing a test situation for the weak kernel),
 we cannot expect to derive an algorithm for 
 solving the lifting problem from these data.
\end{remark}

\begin{remark}\label{remark:abelian_not_computable}
$\Freyd(\Freyd(\Rows_{P})^{\op})$ is
abelian by Theorem \ref{theorem:freyd}.
However, $\Freyd(\Freyd(\Rows_{P})^{\op})$
being computable abelian
implies $P$ having decidable colifts,
which is a computationally undecidable problem.
This behavior is not uncommon in our constructive setup.
Another example of such an abelian category is the category of unbounded chain complexes of
finite dimensional $k$-vector spaces for a field $k$ with decidable equality.
\end{remark}

\section{Lifts and homomorphism structures}\label{section:lifts}

We are going to address the problem of computing lifts in 
the Freyd category $\mathcal{A}(\PC)$ for an additive category $\PC$.
To this end, we take a look at the following diagram:
\begin{center}
        \begin{tikzpicture}[label/.style={postaction={
          decorate,
          decoration={markings, mark=at position .5 with \node #1;}},
          mylabel/.style={thick, draw=none, align=center, minimum width=0.5cm, minimum height=0.5cm,fill=white}}]
          \coordinate (r) at (4.5,0);
          \coordinate (u) at (0,2);
          \node (m) {$(B \stackrel{\rho_B}{\longleftarrow} R_B)$};
          \node (n) at ($(m)+(r)$) {$(C \stackrel{\rho_C}{\longleftarrow} R_C)$.};
          \node (r1) at ($(m) + (u)$) {$(A \stackrel{\rho_A}{\longleftarrow} R_A)$};
          \node (al) at ($(r1) - 0.15*(r) - 0.1*(u)$) {};
          \node (bl) at ($(m) - 0.15*(r) + 0.1*(u)$) {};
          \node (br) at ($(m) + 0.15*(r) + 0.1*(u)$) {};
          \draw[->,thick] (n) --node[below]{$\{ \gamma, \rho_{\gamma} \}$} (m);
          \draw[->,thick,dashed] (r1) --node[above, inner sep = 0.7em]{$\{ X,Y \}$} (n);
          \draw[->,thick] (al) --node[left]{$\{ \alpha, \rho_{\alpha} \}$} (bl);
          \draw[->,thick,dotted,label={[right]{$Z$}}] (al) to (br);
    \end{tikzpicture}
  \end{center}
The solid arrows represent given morphisms in $\mathcal{A}(\PC)$,
the dashed morphism $\{X,Y\}$ in $\mathcal{A}(\PC)$ is the lift
which we want to compute (with unknowns $X,Y$), i.e., it has to satisfy
\begin{equation}\label{equation:welldef}
 \rho_A \cdot X = Y \cdot \rho_C
\end{equation}
for being well-defined and
\begin{equation}\label{equation:lift}
 X \cdot \gamma - \alpha = Z \cdot \rho_B
\end{equation}
for being a lift, where the unknown $Z$ is a witness for $\alpha$ and $X \cdot \gamma$ being equal.
Thus, computing a lift in $\mathcal{A}(\PC)$
means finding morphisms $X,Y,Z$ in $\PC$ satisfying equations \eqref{equation:welldef} and \eqref{equation:lift}.

\subsection{Linear systems in additive categories}

In order to develop a strategy for solving equations \eqref{equation:welldef} and \eqref{equation:lift},
we discuss arbitrary linear systems in $\PC$.

\begin{definition}\label{definition:linear_system}
 Let $\PC$ be an additive category.
 A \textbf{linear system in $\PC$}
 with $m \in \N$ equations in $n \in \N$ indeterminates
 is defined by the following data:
 \begin{enumerate}
  \item Objects $(A_i)_i, (D_i)_i$ and $(B_j)_j, (C_j)_j$ in $\PC$ for $i = 1, \dots, m$, $j = 1, \dots, n$.
  \item Morphisms $(\alpha_{ij}: A_i \rightarrow B_j)_{ij}$ and $(\beta_{ij}: C_j \rightarrow D_i)_{ij}$ in $\PC$ for $i = 1, \dots, m$, $j = 1, \dots, n$.
  \item Morphisms $(\gamma_i: A_i \rightarrow D_i)_i$ in $\PC$ for $i = 1, \dots, m$.
 \end{enumerate}
 A \textbf{solution} is given by morphisms $(X_j: B_j \rightarrow C_j)_{j=1,\dots,n}$
 such that the equations
\[
 \begin{array}{ccccccccc} 
 \alpha_{11} \cdot X_1 \cdot \beta_{11} &+& \alpha_{12} \cdot X_2 \cdot \beta_{12} &+& \dots &+& \alpha_{1n} \cdot X_n \cdot \beta_{1n} &=& \gamma_1 \\
 \vdots & & \vdots& & \vdots& & \vdots& &\vdots \\
 \alpha_{m1} \cdot X_1 \cdot \beta_{m1} &+& \alpha_{m2} \cdot X_2 \cdot \beta_{m2} &+& \dots &+& \alpha_{mn} \cdot X_n \cdot \beta_{mn} &=& \gamma_m \\
 \end{array}
\]
hold.
We say $\PC$ has \textbf{decidable linear systems}
if we have an algorithm that constructs for a given linear system
a solution or disproves its existence.
\end{definition}

\begin{definition}
 Let $\PC$ be an additive category.
 The set of \textbf{iterated Freyd categories of $\PC$}
 is defined inductively:
 \begin{enumerate}
  \item $\PC$ is an iterated Freyd category of $\PC$.
  \item If $\XC$ is an iterated Freyd category of $\PC$, then so are $\XC^{\op}$ and $\Freyd( \XC )$.
 \end{enumerate}
\end{definition}

Important examples of iterated Freyd categories are $\Freyd( \Freyd( \PC ) )$ and $\Freyd( \Freyd( \PC )^{\op} )$
(see Example \ref{example:fpfunctors}). The next theorem
is a generalization of the discussion in the beginning of this section.

\begin{theorem}\label{theorem:linear_system_in_Freyd}
 Let $\PC$ be an additive category.
 Any linear system in an iterated Freyd category $\XC$ of $\PC$
 gives rise to a linear system in $\PC$ 
 such that the former has a solution if and only if the latter has a solution. 
\end{theorem}
\begin{proof}[Proof by induction]
 The case $\XC = \PC$ is trivial.
 Furthermore, any linear system in $\XC^{\op}$ trivially gives rise to an equivalent linear system in $\XC$.
 So,
 let 
\[
 \begin{array}{ccc} 
 \sum_{j=1}^n \{\alpha_{1j}, \rho_{\alpha_{1j}} \} \cdot X_j \cdot \{ \beta_{1j}, \rho_{\beta_{1j}} \}  &=& \{\gamma_1, \rho_{\gamma_1} \} \\
 \vdots & &\vdots \\
 \sum_{i=j}^n \{\alpha_{mj}, \rho_{\alpha_{mj}} \} \cdot X_j \cdot \{\beta_{mj}, \rho_{\beta_{mj}}\}  &=& \{\gamma_m, \rho_{\gamma_m} \}\\
 \end{array}
\]
be a linear system in $\Freyd( \XC )$,
where the corresponding sources and ranges are denoted by
\begin{itemize}
 \item $\{\alpha_{ij}, \rho_{\alpha_{ij}}\}: (A_i \stackrel{\rho_{A_i}}{\longleftarrow} R_{A_i}) \longrightarrow (B_j \stackrel{\rho_{B_j}}{\longleftarrow} R_{B_j})$,
 \item $\{\beta_{ij}, \rho_{\beta_{ij}}\}: (C_j \stackrel{\rho_{C_j}}{\longleftarrow} R_{C_j}) \longrightarrow (D_i \stackrel{\rho_{D_i}}{\longleftarrow} R_{D_i})$,
 \item $\{\gamma_{i}, \rho_{\gamma_{i}}\}: (A_i \stackrel{\rho_{A_i}}{\longleftarrow} R_{A_i}) \longrightarrow (D_i \stackrel{\rho_{D_i}}{\longleftarrow} R_{D_i})$.
\end{itemize}
To define an equivalent linear system in $\XC$,
we introduce variables $(X_j^1)_j$, $(X_j^2)_j$,
and the set of linear equations
\[
 X_j^2 \cdot \rho_{C_j} = \rho_{B_j} \cdot X_j^1
\]
encoding well-definedness of $\{X_j^1, X_j^2\}: (B_j \stackrel{\rho_{B_j}}{\longleftarrow} R_{B_j}) \longrightarrow (C_j \stackrel{\rho_{C_j}}{\longleftarrow} R_{C_j})$.
Furthermore, for each original equation, we need a ``witness-variable'' $(Z_i: A_i \rightarrow R_{D_i})_i$.
Now, we can simply encode the original equations
by the following linear equations in $\XC$:
\[
 \begin{array}{ccc} 
 \sum_{j=1}^n \alpha_{1j} \cdot X_j^1 \cdot \beta_{1j}  &=& \gamma_1 + Z_1 \cdot \rho_{D_1} \\
 \vdots & &\vdots \\
 \sum_{i=j}^n \alpha_{mj} \cdot X_j^1 \cdot \beta_{mj}  &=& \gamma_m + Z_m \cdot \rho_{D_m}. \\
 \end{array}
\]
\end{proof}

\textbf{Key observation:}
finding a solution of a linear system in $\PC$
is equivalent to finding a lift of the following
diagram of abelian groups:
\begin{center}
        \begin{tikzpicture}[label/.style={postaction={
          decorate,
          decoration={markings, mark=at position .5 with \node #1;}},
          mylabel/.style={thick, draw=none, align=center, minimum width=0.5cm, minimum height=0.5cm,fill=white}}]
          \coordinate (r) at (7.5,0);
          \coordinate (u) at (0,2);
          \node (Z) {$\Z$};
          \node (AB) at ($(Z)-(u)$) {$\bigoplus_{i}\Hom_{\PC}( A_i, D_i )$};
          \node (AC) at ($(AB) + (r)$) {$\bigoplus_{j}\Hom_{\PC}( B_j, C_j )$.};
          
          \draw[->,thick] (Z) to node[left]{$(1 \mapsto (\gamma_i)_i)$} (AB);
          \draw[->,thick] (AC) to node[below]{$\big(\Hom_{\PC}(\alpha_{ij}, \beta_{ij})\big)_{ij}$}(AB);
          \draw[->,thick,dashed] (Z) to (AC);
    \end{tikzpicture}
  \end{center}
The data needed to form this diagram
are
\begin{enumerate}
 \item the abelian group $\Z$,
 \item the $\Hom$-functor of $\PC$ mapping to the category of abelian groups,
 \item the translation of elements in $\Hom_{\PC}(A,B)$ to morphisms $\Z \rightarrow \Hom_{\PC}(A,B)$.
\end{enumerate}
We abstract these data
in the following definition\footnote{The author is not aware of such a definition appearing in the literature.}.

\begin{definition}\label{definition:homomorphism_strucutre}
 Let $\PC$, $\BC$ be additive categories.
 A \textbf{$\BC$-homomorphism structure} for $\PC$ consists of the following data:
 \begin{enumerate}
  \item A distinguished object $1 \in \BC$.
  \item A bilinear functor $H: \PC^{\op} \times \PC \rightarrow \BC$, i.e., a functor which is additive in each component.
  \item An isomorphism $\nu: \Hom_{\PC}( P, Q ) \xrightarrow{\sim} \Hom_{\BC}( 1, H( P, Q ) )$ natural in $P,Q \in \PC$,
        i.e, $\nu( \alpha \cdot X \cdot \beta ) = \nu( X ) \cdot H(\alpha, \beta)$ for all composable triples of
        morphisms $\alpha, X, \beta$.
 \end{enumerate}
\end{definition}

\begin{example}\label{example:hom_struct_comm}
We use the notation of Example \ref{example:fpmod}.
Let $R$ be a \emph{commutative} ring.
The functor
\begin{align*}
 H: \Rows_R^{\op} \times \Rows_R &\longrightarrow \Rows_{R}:  \\ 
 (R^{1 \times a} \stackrel{A}{\longleftarrow} R^{1 \times a'}, R^{1 \times b} \stackrel{B}{\longrightarrow} R^{1 \times b'})
 &\longmapsto
 \big(A_{ij}B_{kl}\big)_{ijkl}: \bigoplus_{j=1}^{a}\bigoplus_{k=1}^{b}{R^{1 \times 1}} \longrightarrow \bigoplus_{i=1}^{a'}\bigoplus_{l=1}^{b'}{R^{1 \times 1}}
\end{align*}
defines a $\Rows_R$-homomorphism structure for $\Rows_R$ with distinguished object $R^{1 \times 1}$
and the natural isomorphism of $R$-modules
\[\Hom_{\Rows_R}( R^{1 \times a}, R^{1 \times b} ) \simeq_R \Hom_{\Rows_R}( R^{1 \times 1}, \bigoplus_{j=1}^{a}\bigoplus_{k=1}^{b}{R^{1 \times 1}} ).\]
Note that $H$ can be interpreted as the restriction of the $\Hom$-functor
$R\Modl^{\op} \times R\Modl \rightarrow R\Modl$ to row modules.
\end{example}

The previous example only worked due to the commutativity of the ring
in question.
In the next example, we show what we can do in the non-commutative
case provided that the the center of the ring is ``big enough''.

\begin{example}\label{example:hom_struct_noncomm}
We use the notation of Example \ref{example:fpmod}.
Let $R$ be a ring and $C \subseteq R$ its center.
Assume that $R$ is finitely presented as a $C$-module,
i.e., there exists
an exact sequence of $C$-modules
\begin{center}
    \begin{tikzpicture}[label/.style={postaction={
    decorate,
    decoration={markings, mark=at position .5 with \node #1;}},
    mylabel/.style={thick, draw=none, align=center, minimum width=0.5cm, minimum height=0.5cm,fill=white}}]
    \coordinate (r) at (2.5,0);
    \coordinate (u) at (0,2);
    \node (D) {$0$};
    \node (C) at ($(D) + 0.75*(r)$) {${_{C}R}$};
    \node (B) at ($(C)+(r)$) {$C^{1\times a}$};
    \node (A) at ($(B)+(r)$) {$C^{1 \times b}$};
    
    \draw[->,thick] (A) -- (B);
    \draw[->,thick] (B) -- (C);
    \draw[->,thick] (C) -- (D);
    \end{tikzpicture}
\end{center}
for $a,b \in \N$, where ${_{C}R}$ denotes $R$ regarded as a $C$-module.
In this case the $\Hom$-functor for $\Rows_R$ can be seen as a functor mapping to $C\fpmodl$:
\begin{align*}
 \Rows_R^{\op} \times \Rows_R &\longrightarrow C\fpmodl:  \\ 
 (R^{1 \times c}, R^{1 \times d})
 &\longmapsto \Hom_R( R^{1 \times c}, R^{1 \times d}) \simeq {_{C}R}^{1 \times cd}
\end{align*}
This gives rise to a $C\fpmodl$-homomorphism structure
for $\Rows_R$ with distinguished object $C^{1 \times 1} \in C\fpmodl$
and natural isomorphism
\[\Hom_{\Rows_R}( R^{1 \times c}, R^{1 \times d}) \simeq_C \Hom_{C\fpmodl}\left( C^{1\times 1}, {_{C}R}^{1 \times cd} \right)\]
induced by the natural bijection between elements in $R$ and $C$-module homomorphisms $C \rightarrow {_{C}R}$.
This homomorphism structure transfers to 
a $\Freyd( \Rows_R )$-homomorphism structure
via the equivalence $R\fpmodl \simeq \Freyd( \Rows_R )$ (see Example \ref{example:fpmod}).

\end{example}

The next examples are more abstract.

\begin{example}
 If $\PC$ is an additive closed symmetric monoidal category,
 then its tensor unit $1 \in \PC$ and the internal $\Hom$-functor
 define a $\PC$-homomorphism structure for $\PC$.
\end{example}

\begin{example}\label{example:op}
 If $(1,H: \PC^{\op} \times \PC \rightarrow \B,\nu)$
 is a $\BC$-homomorphism structure for $\PC$, then
 swapping components $G: \PC \times \PC^{\op} \xrightarrow{\sim} \PC^{\op} \times \PC \stackrel{H}{\longrightarrow} \B$
 defines a $\BC$-homomorphism structure for $\PC^{\op}$
 with distinguished object $1$.
\end{example}

The following theorem is an abstraction
of a computational trick presented in
\cite{zl} (see also \cite{BR}).

\begin{theorem}\label{theorem:linear_system_in_P}
We use the notation of Definition \ref{definition:linear_system}.
Let $\PC$ be an additive category equipped with a $\BC$-homomorphism structure $(1, H, \nu)$.
Given a linear system in $\PC$,
then $(X_j)_j$ is a solution if and only if
it gives rise to a lift of the following diagram in $\BC$:
 \begin{center}
        \begin{tikzpicture}[label/.style={postaction={
          decorate,
          decoration={markings, mark=at position .5 with \node #1;}},
          mylabel/.style={thick, draw=none, align=center, minimum width=0.5cm, minimum height=0.5cm,fill=white}}]
          \coordinate (r) at (7.5,0);
          \coordinate (u) at (0,2);
          \node (Z) {$1$};
          \node (AB) at ($(Z)-(u)$) {$\bigoplus_{i}H( A_i, D_i )$};
          \node (AC) at ($(AB) + (r)$) {$\bigoplus_{j}H( B_j, C_j )$.};
          
          \draw[->,thick] (Z) to node[left]{$(\nu(\gamma_i))_i$} (AB);
          \draw[->,thick] (AC) to node[below]{$\big(H(\alpha_{ij}, \beta_{ij})\big)_{ij}$}(AB);
          \draw[->,thick,dashed] (Z) to node[above,xshift=1em]{$(\nu(X_j))_j$} (AC);
    \end{tikzpicture}
  \end{center}
\end{theorem}
\begin{proof}
 We compute
 \begin{align*}
  ( v(X_j) )_j \cdot ( H( \alpha_{ij}, \beta_{ij} ) )_{ij} 
  &= ( \sum_j \nu(X_j) \cdot H( \alpha_{ij}, \beta_{ij} ) )_i & \text{matrix multiplication} \\
  &= ( \sum_j \nu( \alpha_{ij} \cdot X_j \cdot \beta_{ij} ) )_i & \text{naturality of $\nu$}\\
  &= (  \nu( \sum_j \alpha_{ij} \cdot X_j \cdot \beta_{ij} ) )_i & \text{linearity of $\nu$}
 \end{align*}
Since $\nu$ is an isomorphism, the last term equals $(\nu(\gamma_i))_i$
if and only if $(X_j)_j$ is a solution.
\end{proof}

\begin{corollary}\label{corollary:computable_lifts_in_Freyd}
Let $\PC$ be an additive category equipped with a $\BC$-homomorphism structure $(1, H, \nu)$.
If $\BC$ has decidable lifts, then
$\PC$ has decidable linear systems.
In particular, any iterated Freyd category of $\PC$
has decidable linear systems.
\end{corollary}
\begin{proof}
 We use Theorem \ref{theorem:linear_system_in_Freyd} and Theorem \ref{theorem:linear_system_in_P}.
\end{proof}

\begin{example}\label{example:fpmod_lifts}
Let $R$ be a \emph{commutative} ring.
Then
$\Rows_R$ has a $\Rows_R$-homomorphism structure (Example \ref{example:hom_struct_comm}).
If $R$ has decidable lifts,
then
$\Freyd( \Rows_R )$ 
and all the other iterated Freyd categories of $\Rows_R$
have decidable linear systems (in particular lifts and colifts) (Corollary \ref{corollary:computable_lifts_in_Freyd}).
\end{example}

Using the previous example, we can deal with more general rings:

\begin{example}\label{example:fpmod_noncom_lifts}
Let $R$ be a ring
that is finitely presented as a module over its center.
Then $\Rows_R$ has an $\Freyd( \Rows_C )$-homomorphism structure (Example \ref{example:hom_struct_noncomm}).
If $C$ has decidable lifts,
then so does $\Freyd( \Rows_C )$ (Example \ref{example:fpmod_lifts}).
It follows that $\Freyd( \Rows_R )$ 
and all the other iterated Freyd categories of $\Rows_R$
have decidable linear systems (in particular lifts and colifts) (Corollary \ref{corollary:computable_lifts_in_Freyd}).
\end{example}

\subsection{Homomorphism structures in Freyd categories}

We show that certain $\BC$-homo\-morphism structures for $\PC$
induce $\BC$-homomorphism structures for $\mathcal{A}(\PC)$.
We will use such an induced homomorphism structure
in Subsection \ref{subsection:nat} to compute the sets of natural
transformations between finitely presented functors.

\begin{construction}\label{construction:hom_structure_freyd}
 Let $\PC$ be an additive category, $\BC$ an abelian category,
 and $H: \PC^{\op} \times \PC \rightarrow \BC$ a bilinear functor.
 We want to construct from these data a bilinear functor
 \[
  H^{\Freyd}: \Freyd(\PC)^{\op} \times \Freyd( \PC ) \longrightarrow \BC
 \]
 extending $H$ (where we think of $\PC$ as a full subcategory of $\Freyd( \PC )$).
 Given morphisms 
 \[\{\alpha, \rho_{\alpha}\}: (A' \stackrel{\rho_{A'}}{\longleftarrow} R_{A'}) \longrightarrow (A \stackrel{\rho_A}{\longleftarrow} R_A)\]
 and
 \[\{\beta, \rho_{\beta}\}: (B \stackrel{\rho_B}{\longleftarrow} R_B) \longrightarrow (B' \stackrel{\rho_{B'}}{\longleftarrow} R_{B'})\]
 in $\mathcal{A}(\PC)$,
 we set $H^{\Freyd}( \{\alpha, \rho_{\alpha}\}, \{\beta, \rho_{\beta}\} )$ as the morphism
 between the kernels in the diagram
 \begin{center}
    \begin{tikzpicture}[label/.style={postaction={
      decorate,
      decoration={markings, mark=at position .5 with \node #1;}},
      mylabel/.style={thick, draw=black, align=center, minimum width=0.5cm, minimum height=0.5cm,fill=white}}]
      \coordinate (r) at (5,0);
      \coordinate (u) at (0,2.5);
      \node (K1) {$\kernel( \overline{H(\rho_A,B)} )$};
      \node (K2) at ($(K1) - (u)$) {$\kernel( \overline{H(\rho_{A'},B')} )$};
      \node (A) at ($(K1) + (r)$) {$\frac{H(A,B)}{\image(H(A,\rho_B))}$};
      \node (A2) at ($(A)+(r)$) {$\frac{H(R_A,B)}{\image(H(R_A,\rho_B))}$};
      \node (B) at ($(K2) + (r)$) {$\frac{H(A',B')}{\image(H(A',\rho_{B'}))}$};
      \node (B2) at ($(B)+(r)$) {$\frac{H(R_{A'},B')}{\image(H(R_{A'},\rho_{B'}))}$};
      \draw[->,thick] (A) --node[above]{$\overline{H(\rho_A,B)}$} (A2);
      \draw[->,thick] (B) --node[above]{$\overline{H(\rho_{A'},B')}$} (B2);
      \draw[->,thick] (A) --node[left]{$\overline{H(\alpha, \beta)}$} (B);
      \draw[->,thick] (A2) --node[right]{$\overline{H(\rho_\alpha, \beta)}$} (B2);
      \draw[right hook->, thick] (K1) -- (A);
      \draw[right hook->, thick] (K2) -- (B);
      \draw[->, dashed, thick] (K1) --node[left]{$H^{\Freyd}( \{\alpha, \rho_{\alpha}\}, \{\beta, \rho_{\beta}\} )$} (K2);
    \end{tikzpicture}
  \end{center}
  where we overline a morphism to address its induced morphism on quotient objects.
\end{construction}
\begin{proof}[Correctness of the construction]
 To see that all $4$ overlined morphisms in the right square are well-defined,
 let us take a look at $H(\rho_A,B)$.
 It induces a well-defined morphism on the quotient objects in question
 if $\image( H(A,\rho_B) \cdot H(\rho_A,B) ) \subseteq \image( H(R_A, \rho_B) )$.
 But this is true since the interchange law for $H$, i.e., its functoriality as a bifunctor, implies
 \begin{align*}
  H(A,\rho_B) \cdot H(\rho_A,B) = H( \rho_A, R_B ) \cdot H(R_A, \rho_B).
 \end{align*}
 We can deal similarly with the remaining $3$ overlined morphisms. 
 
 Clearly, the right square commutes.
 Furthermore, since $H^{\Freyd}( \{\alpha, \rho_{\alpha}\}, \{\beta, \rho_{\beta}\} )$
 is constructed as a morphism between kernels from the right square,
 it is independent of the morphism $\overline{H(\rho_\alpha, \beta)}$
 and thus of the chosen morphism witnesses $\rho_{\alpha}, \rho_{\beta}$.
 
 Now, let $\sigma_A: A' \rightarrow R_A$ and $\sigma_B: B \rightarrow R_{B'}$
 be arbitrary morphisms.
 Then
 $\{\alpha, \rho_{\alpha}\} = \{\alpha + \sigma_A \cdot \rho_A, \rho_{\alpha} + \rho_{A'} \cdot \sigma_A \}$ and
 $\{\beta, \rho_{\beta}\} = \{\beta + \sigma_B \cdot \rho_{B'}, \rho_{\beta} + \rho_{B} \cdot \sigma_B \}$
 and we have to prove that our construction is independent of this choice of representatives.
 To this end, we compute
 \begin{align*}
  &\overline{H(\alpha + \sigma_A \cdot \rho_A, \beta + \sigma_B \cdot \rho_{B'})} \\
  &= \overline{H(\alpha, \beta)}
  + \overline{H(\sigma_A \cdot \rho_A, \beta )}
  + \underbrace{\overline{H(\alpha ,\sigma_B \cdot \rho_{B'})}}_{= 0}
  + \underbrace{\overline{H(\sigma_A \cdot \rho_A, \sigma_B \cdot \rho_{B'})}}_{= 0}
 \end{align*}
 in $\frac{H(A',B')}{\image(H(A',\rho_{B'}))}$, where the last $2$ summands are $0$ because the images of
 $H(\alpha ,\sigma_B \cdot \rho_{B'})$ and $H(\sigma_A \cdot \rho_A, \sigma_B \cdot \rho_{B'})$
 lie in $\image(H(A',\rho_{B'}))$.
 And since $\overline{H(\sigma_A \cdot \rho_A, \beta )} = \overline{H(\rho_A , B )} \cdot \overline{H( \sigma_A, \beta )}$,
 i.e., it factors over $\overline{H(\rho_A , B )}$,
 this summand does not contribute to the morphism between the kernels.
 It follows that our construction is well-defined, and its
 functoriality and bilinearity are easy to see.
\end{proof}

\begin{theorem}\label{theorem:induced_hom_structure}
 Let $\PC$ be an additive category equipped with a $\BC$-homomorphism structure $(1, H, \nu)$.
 If $\BC$ is abelian and if $1 \in \BC$ is a projective object,
 then the functor $H^{\Freyd}$ of 
 Construction \ref{construction:hom_structure_freyd}
 fits into a $\BC$-homomorphism structure
 $(1, H^{\Freyd}, \nu^{\Freyd})$ for $\Freyd(\PC)$.
\end{theorem}
\begin{proof}
 We are using the notation of Construction \ref{construction:hom_structure_freyd}.
 Since $1$ is projective, $\Hom_{\BC}(1,-)$ is exact and thus
 commutes with all abelian constructions.
 From this, it follows that $\Hom_{\BC}\big(1, \kernel( \overline{H(\rho_A,B)} ) \big)$
 is naturally isomorphic to the kernel of the morphism between abelian groups
 \begin{center}
    \begin{tikzpicture}[label/.style={postaction={
      decorate,
      decoration={markings, mark=at position .5 with \node #1;}},
      mylabel/.style={thick, draw=black, align=center, minimum width=0.5cm, minimum height=0.5cm,fill=white}}]
      \coordinate (r) at (6,0);
      \coordinate (u) at (0,2.5);
      \node (A) at ($(K1) + (r)$) {$\frac{\Hom_{\PC}(A,B)}{\image(\Hom_{\PC}(A,\rho_B))}$};
      \node (A2) at ($(A)+(r)$) {$\frac{\Hom_{\PC}(R_A,B)}{\image(\Hom_{\PC}(R_A,\rho_B))}$};
      \draw[->,thick] (A) --node[above]{$\overline{\Hom_{\PC}(\rho_A,B)}$} (A2);
    \end{tikzpicture}
  \end{center}
 which readily identifies with $\Hom_{\Freyd(\PC)}( (A \stackrel{\rho_A}{\longleftarrow} R_A), (B \stackrel{\rho_B}{\longleftarrow} R_B))$.
\end{proof}

\begin{corollary}\label{corollary:iterated_hom_structures}
 Let $\PC$ be an additive category equipped with a $\BC$-homomorphism structure $(1, H, \nu)$.
 If $\BC$ is abelian and if $1 \in \BC$ is a projective object,
 then all iterated Freyd categories $\XC$ of $\PC$
 can be equipped with a $\BC$-homomorphism structure.
\end{corollary}
\begin{proof}[Proof by induction]
 The case $\XC = \PC$ is trivial. The case $\XC^{\op}$ follows from Example \ref{example:op}.
 The case $\Freyd( \XC )$ is Theorem \ref{theorem:induced_hom_structure}.
\end{proof}

We summarize the computationally best case for rings.

\begin{corollary}\label{corollary:fpmod}
Let $R$ be a left coherent ring that is finitely presented as a module over its center $C$.
If $C$ is computable,
then all iterated Freyd categories $\XC$ of $\Freyd( \Rows_R )$
are computable abelian, have decidable linear systems (in particular lifts and colifts),
and have an $\Freyd( \Rows_C )$-homomorphism structure.
\end{corollary}
\begin{proof}
Use Corollary \ref{corollary:computable_lifts_in_Freyd}, Corollary \ref{corollary:iterated_hom_structures},
and Theorem \ref{theorem:rings_and_freyd}.
Also note that the assumptions imply that $R$ is in fact a left computable ring.
\end{proof}

We close this section with an example
of a ring that does not meet the assumptions of Corollary \ref{corollary:fpmod}.

\begin{theorem}\label{theorem:ring_non_coherent_but_iterated_freyd_computable}
 Let $k$ be a field with decidable equality and set
 \[
 R := k[ z, x_i \mid i \in \N ] / \langle zx_i \mid i \in \N \rangle.
 \]
 Then $R$ has decidable lifts, is not left coherent,
 but
 $\Freyd( \Freyd( \Rows_R )^{\op} )$ is a computable abelian category
 that has decidable linear systems.
\end{theorem}
\begin{proof}
 $R$ is not left coherent since the
 the kernel (in $R\Modl$) of
 $R^{1 \times 1 } \stackrel{(z)}{\longrightarrow} R^{1 \times 1}$
 is given by $\langle x_i \mid i \in \N \rangle \leq R^{1 \times 1}$
 and a simple degree argument shows that it cannot be generated by finitely many elements.
 Next, we show that $\Rows_R$ has decidable lifts.
 For any finite subset $M \subseteq \{x_i \mid i \in \N\}$ of variables,
 define $R_M := k[z,m \mid m \in M]/ \langle zm \mid m \in M \rangle$.
 The following ring homomorphisms
 \begin{center}
    \begin{tikzpicture}[label/.style={postaction={
      decorate,
      decoration={markings, mark=at position .5 with \node #1;}},
      mylabel/.style={thick, draw=black, align=center, minimum width=0.5cm, minimum height=0.5cm,fill=white}}]
      \coordinate (r) at (5,0);
      \coordinate (u) at (0,2.5);
      \node (A) {$R_M$};
      \node (B) at ($(A)+(r)$) {$R$};
      \node (C) at ($(B)+(r)$) {$R_M$};
      \draw[right hook->,thick] (A) --node[below]{$ 
              \begin{array}{cc}
                 z \mapsto z  \\
                 m \mapsto m 
              \end{array}$} (B);
      \draw[->>,thick] (B) --node[below]{$ 
              \begin{array}{cc}
                 z \mapsto& \hspace{-7.5em}z  \\
                 x_i \mapsto&
                 \left\{ 
                    \begin{array}{cc}
                 x_i & x_i \in M  \\
                 0 & \text{else}  
                  \end{array} 
                 \right.
              \end{array}$} (C);
      \draw[bend left,->,thick,label={[above]{$\id$}},out=25,in=155] (A) to (C);
    \end{tikzpicture}
  \end{center}
 give rise to the functor $- \otimes_{R_M} R: \Rows_{R_M} \rightarrow \Rows_{R}$
 that interprets entries in $R_M$ as entries in $R$,
 and to the functor $- \otimes_{R} R_M: \Rows_{R} \rightarrow \Rows_{R_M}$
 that replaces every $x_i \not\in M$ with $0$.
 
 Let $R^{1 \times a} \stackrel{A}{\longrightarrow} R^{1 \times b} \stackrel{C}{\longleftarrow} R^{1 \times c}$
 be a cospan for $a,b,c \in \N_0$. We define a particular finite subset $M$ by taking all variables $x_i$
 that occur in the representatives of the entries in $A$ and $C$.
 Due to this choice, we have $(A \otimes_{R} R_M) \otimes_{R_M} R = A$ and $(C \otimes_{R} R_M) \otimes_{R_M} R = C$.
 It follows that any lift $L$ of $(A \otimes_{R} R_M)$ along $(C \otimes_{R} R_M)$
 yields a lift $L \otimes_{R_M} R$ of $A$ along $C$.
 Conversely, any lift $L'$ of $A$ along $C$ yields a lift $L' \otimes_{R} R_M$ of $A\otimes_{R} R_M$ along $C\otimes_{R} R_M$.
 Since $R_M$ is a computable ring (by means of Gröbner bases \cite{GP}), $\Rows_R$ has decidable lifts.
 
 It follows from Example \ref{example:fpmod_lifts}
 that $\Freyd( \Rows_R )^{\op}$ and $\Freyd( \Freyd( \Rows_R )^{\op} )$ have decidable linear systems.
 Furthermore, $\Freyd( \Rows_R )^{\op}$ has kernels since $\Freyd( \Rows_R )$ has cokernels (by Construction \ref{construction:cokernel}).
 Now, it follows from Corollary \ref{corollary:freyd} that $\Freyd( \Freyd( \Rows_R )^{\op} )$ is a computable abelian category.
\end{proof}

\section{Applications to finitely presented functors}\label{section:applications_fpfunctors}

For an additive category $\AC$ 
we have the classical interpretations
\begin{equation}\label{equation:freyd_fp_functors_id}
 \Freyd( \AC ) \simeq \fp( \AC^{\op}, \Ab ) \hspace{1em} \text{and} \hspace{1em} \Freyd( \AC^{\op} ) \simeq \fp( \AC, \Ab )
\end{equation}
of the Freyd categories as categories of finitely presented (contravariant) functors
(see Example \ref{example:fpfunctors}).
The categories $\fp( \AC^{\op}, \Ab )$ and $\fp( \AC, \Ab )$
in the case $\AC$ abelian
were extensively studied by Auslander in \cite{A}.
The goal of this section is to
benefit from our constructive approach to Freyd categories
for the study of finitely presented functors via the equivalence \eqref{equation:freyd_fp_functors_id}:
we describe constructions on the level of Freyd categories
and interpret them in terms of (classical) notions and constructions of finitely presented functors.

\subsection{Ext and Tor}

\begin{example}\label{example:exts}
Let $\AC$ be an abelian category with enough projectives.
Then for all $A \in \AC$ and $i \geq 0$, the functors
$\Ext^i(A,-)$ are finitely presented \cite{A}.
We describe a way to construct them as objects in $\mathcal{A}(\AC^{\op})$.
To any chain complex 
\begin{center}
    \begin{tikzpicture}[label/.style={postaction={
    decorate,
    decoration={markings, mark=at position .5 with \node #1;}},
    mylabel/.style={thick, draw=none, align=center, minimum width=0.5cm, minimum height=0.5cm,fill=white}}]
    \coordinate (r) at (2.5,0);
    \coordinate (u) at (0,2);
    \node (A) {$A_{\bullet}:$};
    \node (dots1) at ($(A)+0.5*(r)$) {$\dots$};
    \node (B) at ($(dots1)+(r)$) {$A_{i-1}$};
    \node (C) at ($(B) + (r)$) {$A_{i}$};
    \node (D) at ($(C) + (r)$) {$A_{i+1}$};
    \node (dots2) at ($(D)+(r)$) {$\dots$};
    \draw[->,thick] (B) -- (dots1);
    \draw[->,thick] (C) --node[above]{$d^A_i$} (B);
    \draw[->,thick] (D) --node[above]{$d^A_{i+1}$} (C);
    \draw[->,thick] (dots2) -- (D);
    \end{tikzpicture}
\end{center}
in $\AC$, we can associate the object $(\kernel(d^A_{i}) \hookrightarrow A_i) \in \mathcal{A}(\AC^{\op})$
and this association is easily seen to define a contravariant functor $\ker_i$ from
the category of chain complexes modulo homotopy $K_{\bullet}(\AC)$ to $\mathcal{A}(\AC^{\op})$.
Now, let $\ProjRes: \AC \rightarrow K_{\bullet}(\AC)$
denote the functor sending an object $A\in \AC$ to its projective resolution $(P_{\bullet},d^P_{\bullet})$
(which is uniquely determined up to homotopy).
Then the functor
\[
 \AC^{\op} \longrightarrow \fp( \AC, \Ab): A \mapsto \Ext^i(A,-)
\]
corresponds to 
\[
\ker_i \circ \ProjRes^{\op}: \AC^{\op} \longrightarrow \Freyd( \AC^{\op} ): A \mapsto (\Omega^i A \hookrightarrow P_{i-1})
\]
via the equivalence \eqref{equation:freyd_fp_functors_id},
where $\Omega^i A := \kernel( d^P_{i-1} )$ denotes the $i$-th syzygy object.
Note that this correspondence is simply given by the
description of $\Ext^i$ in terms of an $i$-th right satellite:
\begin{align*}
 \Ext^i(A,-)(B) & \simeq S^i(-,B)(A) \\
                & \simeq \cokernel\left( (\Omega^i A, B ) \longleftarrow ( P_{i-1}, B ) \right) 
\end{align*}
natural in $A,B \in \AC$ (see \cite[Chapter 3]{ce} for an explanation of satellites).
\end{example}

\begin{example}\label{example:tors}
 Let $R$ be a ring and $M$ be a right $R$-module.
 Then it is shown in \cite[Lemma 6.1]{A}
 that the functor $(M \otimes_R -): R\Modl \rightarrow \Ab$
 is finitely presented if and only if $M$ is a finitely presented module.
 Concretely, if $0 \longleftarrow M \longleftarrow P_0 \longleftarrow P_1$ is a 
 presentation with $P_0, P_1$ finitely presented projective modules, then
 right exactness of the tensor product implies exactness of the rows in
 \begin{center}
    \begin{tikzpicture}[label/.style={postaction={
    decorate,
    decoration={markings, mark=at position .5 with \node #1;}},
    mylabel/.style={thick, draw=none, align=center, minimum width=0.5cm, minimum height=0.5cm,fill=white}}]
    \coordinate (r) at (4,0);
    \coordinate (u) at (0,1.5);
    \node (A) {$0$};
    \node (B) at ($(A)+(r)$) {$(M \otimes_R -)$};
    \node (C) at ($(B)+(r)$) {$(P_0 \otimes_R -)$};
    \node (D) at ($(C) + (r)$) {$(P_1 \otimes_R -)$};
    \node (A2) at ($(A) - (u)$) {$0$};
    \node (B2) at ($(A2)+(r)$) {$(M \otimes_R -)$};
    \node (C2) at ($(B2)+(r)$) {$(P_0^{\vee},-)$};
    \node (D2) at ($(C2) + (r)$) {$(P_1^{\vee},-)$,};
    \draw[->,thick] (B) -- (A);
    \draw[->,thick] (C) -- (B);
    \draw[->,thick] (D) -- (C);
    
    \draw[draw=none] (B) --node[above,rotate=90,inner sep=0em]{$=$} (B2);
    \draw[->,thick] (C) --node[above,rotate=90,inner sep=0em]{$\sim$} (C2);
    \draw[->,thick] (D) --node[above,rotate=90,inner sep=0em]{$\sim$} (D2);
    
    \draw[->,thick] (B2) -- (A2);
    \draw[->,thick] (C2) -- (B2);
    \draw[->,thick] (D2) -- (C2);
    \end{tikzpicture}
\end{center}
 where $(-)^{\vee}$ denotes the dualization $\Hom_R(-,R)$.
 Thus, $(M \otimes_R -)$ corresponds via \eqref{equation:freyd_fp_functors_id} to 
 \[(P_0^{\vee} \longrightarrow P_1^{\vee}) \in \Freyd( R\Modl^{\op} ).\]
 
 Next, we discuss the $\Tor$ functors.
 We set $\Tr(M) := \cokernel( P_0^{\vee} \longrightarrow P_1^{\vee} )$.
 Then there is an exact sequence of the form
 \begin{equation}\label{equation:tor_seq}
  0 \longleftarrow \Tor_1(M,-) \longleftarrow (\Tr(M),-) \longleftarrow (\Tr(M)^{\vee}\otimes_R -)
 \end{equation}
 due to \cite[Proposition 6.3]{A} and \cite[Proposition 7.1]{A}.
 Let us translate this sequence to $\Freyd( R\Modl^{\op} )$ in the case where $\Tr(M)^{\vee}$ is finitely presented
 (e.g., when $R$ is right coherent).
 Let 
 \[
 0 \longleftarrow \Tr(M)^{\vee} \stackrel{\varepsilon}{\longleftarrow} Q_0 \stackrel{q}{\longleftarrow} Q_1
 \]
 be a presentation with $Q_0, Q_1$ finitely presented projective right modules.
 Then a short computation shows that the following diagram
 \begin{center}
    \begin{tikzpicture}[label/.style={postaction={
    decorate,
    decoration={markings, mark=at position .5 with \node #1;}},
    mylabel/.style={thick, draw=none, align=center, minimum width=0.5cm, minimum height=0.5cm,fill=white}}]
    \coordinate (r) at (3.5,0);
    \coordinate (u) at (0,2);
    \node (Z) {$0$};
    \node (T1) at ($(Z) + (r)$) {$\Tor_1(M,-)$};
    \node (TrM) at ($(T1) + (r)$) {$(\Tr(M),-)$};
    \node (TrMd) at ($(TrM) + 2*(r)$) {$(\Tr(M)^{\vee}\otimes_R - )$};
    
    \node (Q0) at ($(TrMd) + (u)$) {$(Q_0 \otimes_R -)$};
    \node (Q0d) at ($(TrMd) + (u) - (r)$) {$(Q_0^{\vee},-)$};
    
    \node (Q1) at ($(Q0) + (u)$) {$(Q_1 \otimes_R -)$};
    \node (Q1d) at ($(Q0d) + (u)$) {$(Q_1^{\vee},-)$};
    
    \node (Z2) at ($(TrMd) - (u)$) {$0$};
    
    \draw[->,thick] (T1) -- (Z);
    \draw[->,thick] (TrM) -- (T1);
    \draw[->,thick] (TrMd) -- (TrM);
    
    \draw[->,thick] (Q0) -- (TrMd);
    \draw[->,thick] (Q1) -- (Q0);
    \draw[->,thick] (TrMd) -- (Z2);
    
    \draw[->,thick] (Q0) --node[above,inner sep = 0em]{$\sim$} (Q0d);
    \draw[->,thick] (Q1) --node[above,inner sep = 0em]{$\sim$} (Q1d);
    
    \draw[->,thick] (Q1d) --node[right]{$(q^{\vee},-)$} (Q0d);
    
    \draw[->,thick] (Q0d) --node[right,yshift = -0.3em]{$(\psi,-)$} (TrM);
    
    \draw (Q1d) [bend right,->,thick,out=-45]  to node[right]{$0$} (TrM);
    \end{tikzpicture}
\end{center}
 commutes, where 
 \[
 \psi := \Tr(M) \longrightarrow (\Tr(M)^{\vee})^{\vee} \stackrel{\varepsilon^{\vee}}{\longrightarrow} Q_0^{\vee}.
 \]
 It follows that the natural transformation $(\Tr(M),-) \longleftarrow (\Tr(M)^{\vee}\otimes_R -)$
 is induced by the universal property of the cokernel
 of $(Q_0^{\vee},-) \longleftarrow (Q_1^{\vee},-)$.
 But the computation of such a morphism is a task that we can easily translate to $\Freyd( R\Modl^{\op} )$
 using the methods of Section \ref{section:constructive_freyd_categories}.
 We end up with the following translation:
 \begin{center}
    \begin{tikzpicture}[label/.style={postaction={
    decorate,
    decoration={markings, mark=at position .5 with \node #1;}},
    mylabel/.style={thick, draw=none, align=center, minimum width=0.5cm, minimum height=0.5cm,fill=white}}]
    \coordinate (r) at (4,0);
    \coordinate (u) at (0,2);
    \node (A) {$0$};
    \node (B) at ($(A)+(r)$) {$\Tor_1(M,-)$};
    \node (C) at ($(B)+(r)$) {$(\Tr(M),-)$};
    \node (D) at ($(C) + (r)$) {$(\Tr(M)^{\vee}\otimes_R -)$};
    \node (A2) at ($(A) - (u)$) {$0$};
    \node (B2) at ($(A2)+(r)$) {$(\Tr(M) \stackrel{\psi}{\longrightarrow} Q_0^{\vee})$};
    \node (C2) at ($(B2)+(r)$) {$(\Tr(M) {\longrightarrow} 0)$};
    \node (D2) at ($(C2) + (r)$) {$(Q_0^{\vee} \stackrel{q^{\vee}}{\longrightarrow} Q_1^{\vee})$.};
    \node (I1) at ($(B) + (r) - 0.5*(u)$) {$\big\updownarrow$\text{~ translates to}};
    \draw[->,thick] (B) -- (A);
    \draw[->,thick] (C) -- (B);
    \draw[->,thick] (D) -- (C);
    \draw[->,thick] (B2) --node[below]{$\{1,0\}$} (C2);
    \draw[->,thick] (C2) --node[below]{$\{\psi,0\}$} (D2);
    \draw[->,thick] (A2) -- (B2);
    \end{tikzpicture}
\end{center}
Note that since we work with $\Freyd(\PC)$ for $\PC = R\Modl^{\op}$, the arrows in the second sequence point to the right.

For higher $\Tor$s, we can use the isomorphism $\Tor_i(M,-) \simeq \Tor_{1}( \Omega^{i-1}M,-)$ for $i \geq 1$
whenever $\Omega^{i-1}M$ is a finitely presented left module and $\Tr(\Omega^{i-1}M)^{\vee}$
is a finitely presented right module (for example in the case where $R$ is both left and right coherent).

Alternatively, we can represent $\Tor_i(M,-)$ as a finitely presented functor in terms of left satellites:
for $N \in R\Modl$, we have
\begin{align*}
 \Tor_i( M, - )(N) &\simeq S_i( - \otimes N)(M) \\
                   &\simeq \kernel( \Omega^i M \otimes N \longrightarrow P_{i-1} \otimes M ),
\end{align*}
where $\Omega^i M \hookrightarrow P_{i-1}$ denotes the embedding of the $i$-th syzygy object
in the $(i-1)$-th object of a projective resolution of $M$. If $\Omega^i M$ and $P_{i-1}$ are finitely presented as modules,
then so is
\[
 \Tor_i(M,-) \simeq \kernel \left( (\Omega^i M \otimes -) \longrightarrow (P_{i-1} \otimes -) \right)
\]
as a functor (see also \cite[Theorem 10.2.35]{PrestPSL}).
\end{example}

\begin{remark}
 Let $R$ be a ring.
 Then $\fp( R\fpmodl, \Ab )$ can be seen as a 
 full and exact subcategory of $\fp( R\Modl, \Ab)$ by mapping 
 a finitely presented functor on $R\fpmodl$ to its colimit extension
 (see, e.g., the discussion in \cite[Section 4]{MartRussII}).
 Thus, if we use our model $R\fpmodl \simeq \Freyd( \Rows_R )$ described
 in Example \ref{example:fpmod}, then
 \[\Freyd( \Freyd( \Rows_R )^{\op}) \simeq \fp( R\fpmodl, \Ab ) \subseteq \fp( R\Modl, \Ab) \]
 gives a model for explicit computations with those finitely presented functors
 on $R\Modl$ that commute with filtered colimits.
\end{remark}

\subsection{Computing sets of natural transformations}\label{subsection:nat}
If $R$ is a left coherent ring
that is finitely presented as a module over its computable center $C$,
then due to Corollary \ref{corollary:fpmod},
$\Freyd( \Freyd( \Rows_R )^{\op} ) \simeq \fp( R\fpmodl, \Ab )$ 
has an $\Freyd( \Rows_C ) \simeq C\fpmodl$-homomorphism structure $H$.
This enables us to explicitly compute
sets of natural transformations.
For example, 
let $A$ be a finitely presented left module,
$M$ a finitely presented right module
such that $\Tr(M)^{\vee}$ is also finitely presented
(where we use the notation of Example \ref{example:exts}
and Example \ref{example:tors}).
Then we may determine for $i \geq 0$ (in a possibly new way)
\[\Hom_{\fp( R\fpmodl, \Ab )}\big( \Tor_1(M,-), \Ext^i(A,-) \big)\]
by computing
\[
 H\big( (\Tr(M) \stackrel{\psi}{\longrightarrow} Q_0^{\vee}), (\Omega^i A \hookrightarrow P_{i-1}) \big)
\]
as an object $(C^{1 \times m} \longleftarrow C^{1 \times n})$ in $\Freyd( \Rows_C )$, where $m,n \in \N_0$.
This tells us that as a $C$-module, $\Hom_{\fp( R\fpmodl, \Ab )}( \Tor_1(M,-), \Ext^i(A,-) )$
can be presented by $m$ generators and $n$ relations.
Using the definition of an $\Freyd( \Rows_C )$-homomorphism structure,
we can explicitly determine the $m$ morphisms in $\Freyd( \Freyd( \Rows_R )^{\op} )$ 
that correspond to our generators.

Thus, we have realized the idea of $\Hom$-computability
for $\Freyd( \Freyd( \Rows_R )^{\op} )$
in the sense of finding a way to ``oversee'' a $\Hom$-group in
that category. All we needed for the realization of that idea
was the notion of a homomorphism structure,
which is much easier to implement
on a computer than the creation of a constructive setup
for arbitrary enriched categories as it is suggested in
\cite[Appendix]{BL_ExtComputability}.

\subsection{Deciding left exactness}\label{subsection:left}

\begin{theorem}
 Let $\AC$ be an abelian category. An object
 $(A \stackrel{\rho_A}{\longleftarrow} R_A) \in \Freyd(\AC)$
 corresponds to a left exact functor in $\fp( \AC^{\op}, \Ab)$
 (via \eqref{equation:freyd_fp_functors_id})
 if and only if
 \[
  \{\varepsilon,0\}: (A \stackrel{\rho_A}{\longleftarrow} R_A) {\longrightarrow} (\cokernel{\rho_A} {\longleftarrow} 0)
 \]
 is an isomorphism in $\Freyd(\AC)$, where $\varepsilon: A \rightarrow \cokernel{\rho_A}$
 denotes the cokernel projection.
\end{theorem}
\begin{proof}
 Let $F := \cokernel( (-,A) \stackrel{(-,\rho_A)}{\longleftarrow} (-,R_A) )$.
 In \cite[Section 3]{A}, it is shown that
 \[ F\longrightarrow (-,\cokernel{\rho_A}) \] 
 is the universal left exact approximation of $F$,
 i.e., for any left exact $G: \AC^{\op} \rightarrow \Ab$ and natural
 transformation $F \rightarrow G$, there exists 
 exactly one natural transformation making the diagram
 \begin{center}
    \begin{tikzpicture}[label/.style={postaction={
    decorate,
    decoration={markings, mark=at position .5 with \node #1;}},
    mylabel/.style={thick, draw=none, align=center, minimum width=0.5cm, minimum height=0.5cm,fill=white}}]
    \coordinate (r) at (1.5,0);
    \coordinate (u) at (0,1.5);
    \node (A) {$F$};
    \node (C) at ($(A)-(u)+(r)$) {$(-,\cokernel{\rho_A})$};
    \node (B) at ($(C)-2*(r)$) {$G$};
    
    \draw[->,thick] (A) -- (B);
    \draw[->,thick] (A) -- (C);
    \draw[->,thick,dashed] (C) -- (B);
    \end{tikzpicture}
\end{center}
 commutative. From this and the fact that any left exact functor is its own universal left exact approximation, the claim follows.
\end{proof}

\begin{remark}\label{remark:isomorphism}
Deciding if a given morphism $\varphi$ in $\Freyd(\AC)$ is an isomorphism 
can be done
by deciding whether $\id_{\kernel( \varphi )} = 0$ and $\id_{\cokernel( \varphi )} = 0$,
which can be constructively realized with the methods provided
by Section \ref{section:constructive_freyd_categories} and \ref{section:lifts}.
\end{remark}

\subsection{Computing injective resolutions}\label{subsection:inj}

\begin{theorem}\label{theorem:injectives}
 Let $\AC$ be an abelian category having enough projectives.
 Let furthermore $(A \stackrel{\rho_A}{\longrightarrow} R_A) \in \Freyd(\AC^{\op})$,
 $\varepsilon_Q: Q \twoheadrightarrow R_A$ an epimorphism with $Q$ projective,
 $\varepsilon_P: P \twoheadrightarrow A \times_{R_A} Q$ an epimorphism with $P$ projective.
 Set $\rho_P := \varepsilon_P \cdot \secondproj: P \rightarrow Q$.
 Then the morphism 
 \[
  \{ \varepsilon_P \cdot \firstproj, \varepsilon_Q \}: (A \stackrel{\rho_A}{\longrightarrow} R_A) \longleftarrow (P \stackrel{\rho_P}{\longrightarrow} Q) \in \Freyd(\AC^{\op})
 \]
 is a monomorphism with $(P \stackrel{\rho_P}{\longrightarrow} Q)$ an injective object.
\end{theorem}
\begin{proof}
The construction dual to the one described 
in our theorem can be found in
\cite[below Proposition 1.4]{Gentle91},
where Gentle identifies $\fp( \AC^{\op}, \Ab )$
with the category of left exact sequences in $\AC$
modulo chain homotopy, a category which is readily seen
to be equivalent to $\Freyd( \AC )$.
\end{proof}

\subsection{Deciding right exactness}\label{subsection:right}

\begin{theorem}
 Let $\AC$ be an abelian category having enough projectives.
 An object $(A \stackrel{\rho_A}{\longrightarrow} R_A) \in \Freyd(\AC^{\op})$
 corresponds
 to a right exact functor in $\fp( \AC, \Ab)$
 (via \eqref{equation:freyd_fp_functors_id})
 if and only if
 the monomorphism defined in Theorem \ref{theorem:injectives} splits.
\end{theorem}
\begin{proof}
 The monomorphism splits if and only if $(A \stackrel{\rho_A}{\longrightarrow} R_A)$
 is an injective object,
 which is the case if and only if 
 it represents a right exact functor \cite[Lemma 5.1]{A}.
\end{proof}

\begin{remark}
Deciding if a given monomorphism $\varphi: A \rightarrow B$ splits
can be done
by deciding whether there exists a colift of the span $A \stackrel{\id}{\longleftarrow} A \stackrel{\varphi}{\longrightarrow} B$.
We have also seen (Corollary \ref{corollary:fpmod}) that
$\Freyd( \Freyd( \Rows_R )^{\op} )$
has decidable colifts
provided that
$R$ is a left coherent ring
that is finitely presented as a module over its computable center $C$.
\end{remark}

\begin{appendix}

\section{Axioms of categories}\label{appendixa}

We give complete definitions
of various kinds of categories in a constructive context.
We also hint at important differences to the corresponding list in \cite[Appendix B]{BL_GabrielMorphisms}.

\begin{definition}\label{definition:computable_category}
 A \textbf{category} $\AC$ consists of the following data:
 \begin{enumerate}
  \item A data type $\Obj_{\AC}$ 
        (\textbf{objects}).
  \item Depending on $A,B \in \Obj_{\AC}$, a data type $\Hom_{\AC}(A,B)$ (\textbf{morphisms}),
        each equipped with an equivalence relation $=$ (\textbf{equality}).
  \item An algorithm that computes for given $A,B,C \in \Obj_{\AC}$, $\alpha \in \Hom_{\AC}(A,B)$, $\beta \in \Hom_{\AC}(B,C)$
        a morphism $\alpha \cdot \beta \in \Hom_{\AC}(A,C)$
        (\textbf{composition}).
        If $D \in \Obj_{\AC}$ and $\gamma \in \Hom_{\AC}(C,D)$, then
        \[
         (\alpha \cdot \beta)\cdot \gamma = \alpha \cdot ( \beta \cdot \gamma ) \text{~(\textbf{associativity})}.
        \]
  \item An algorithm that constructs for given $A \in \Obj_{\AC}$ a morphism $\id_A \in \Hom_{\AC}(A,A)$
        (\textbf{identities}). For $B, C \in \Obj_{\AC}$, $\beta \in \Hom_{\AC}(B,A)$, $\gamma \in \Hom_{\AC}(A,C)$,
        we have 
        \[
         \beta \cdot \id_A = \beta \text{\hspace{1em} and \hspace{1em}} \id_A \cdot \gamma = \gamma.
        \]
  \end{enumerate}
\end{definition}

\begin{remark}\label{remark:treatment_of_equality}
Note that we treat equality for morphisms as an extra datum
attached to $\Hom_{\AC}(A,B)$ and so we think
of $\Hom_{\AC}(A,B)$ more as a setoid.
This point of view is very convenient
in our treatment of Freyd categories,
since witnesses of two morphisms being equal 
are actually used in concrete constructions and thus cannot be ignored (see Remark \ref{remark:constructiveness}).
Also note that this point is not stressed in the corresponding list of axioms in \cite[Appendix B]{BL_GabrielMorphisms}.
\end{remark}

\begin{definition}
  An \textbf{Ab-category} is a category $\AC$ for which we have:
  \begin{enumerate}
  \setcounter{enumi}{5}
  \item An algorithm that computes for given $A,B \in \Obj_{\AC}$, $\alpha, \beta \in \Hom_{\AC}(A,B)$
        a morphism $\alpha + \beta \in \Hom_{\AC}(A,B)$ (\textbf{addition}).
  \item An algorithm that constructs for given $A,B \in \Obj_{\AC}$ a morphism $0 \in \Hom_{\AC}(A,B)$
        (\textbf{zero morphisms}).
  \item An algorithm that computes for given $A,B \in \Obj_{\AC}$, $\alpha \in \Hom_{\AC}(A,B)$ 
        a morphism $-\alpha \in \Hom_{\AC}(A,B)$ (\textbf{additive inverse}).
  \item For $A,B,C \in \Obj_{\AC}$, $\alpha, \beta, \gamma \in \Hom_{\AC}(A,B)$, $\delta, \epsilon \in \Hom_{\AC}(B,C)$, we have
        \begin{enumerate}
         \item $\alpha + \beta = \beta + \alpha$ (\emph{commutativity}),
         \item $(\alpha + \beta) + \gamma = \alpha + (\beta + \gamma)$ (\emph{associativity}),
         \item $0 + \alpha = \alpha + 0 = \alpha$ (\emph{neutral element}),
         \item $\alpha + (-\alpha) = (-\alpha) + \alpha = 0$ (\emph{inverse element}),
         \item $\alpha \cdot (\delta + \epsilon) = \alpha \cdot \delta + \alpha \cdot \epsilon$ (\emph{left distributivity}),
         \item $(\alpha + \beta) \cdot \delta = \alpha \cdot \delta + \alpha \cdot \delta$ (\emph{right distributivity}).
        \end{enumerate}
\end{enumerate}
\end{definition}

\begin{definition}
  An \textbf{additive category} is an Ab-category $\AC$ for which we have:
  \begin{enumerate}
  \setcounter{enumi}{9}
  \item An algorithm that computes for a given finite (possibly empty) list of objects $A_1, \dots, A_n$ in $\Obj_{\AC}$ (for $n \in \N_0$)
        an object $\bigoplus_{i=1}^n A_i \in \Obj_{\AC}$ (\textbf{direct sum}).
        If we are additionally given an integer $i \in \{ 1 \dots n\}$,
        we furthermore have algorithms for computing 
        morphisms $\pi_i \in \Hom_{\AC}(\bigoplus_{i=1}^n A_i, A_i)$ (\textbf{direct sum projection})
        and $\iota_i: \Hom_{\AC}(A_i, \bigoplus_{i=1}^n A_i)$ (\textbf{direct sum injection}).
  \item The identities
        $\sum_{i=1}^n \pi_i \cdot \iota_i = \id_{\bigoplus_{i=1}^n A_i}$,
        $\iota_i \cdot \pi_i = \id_{A_i}$,
        and 
        $\iota_i \cdot \pi_j = 0$,
        hold for all $i,j = 1, \dots, n$, $i \not= j$.
 \end{enumerate}
\end{definition}

\begin{definition}
A \textbf{preabelian category} is an additive category $\AC$ for which we have:
 \begin{enumerate}
  \setcounter{enumi}{11}
  \item Algorithms that compute for given $A,B \in \Obj_{\AC}$, $\alpha \in \Hom_{\AC}(A,B)$ an object $\kernel( \alpha ) \in \Obj_{\AC}$ (\textbf{kernel object})
        and a morphism 
        \[\KernelEmbedding( \alpha ) \in \Hom_{\AC}(\kernel( \alpha ), A) \text{\hspace{1em}(\textbf{kernel embedding})}\]
        for which $\KernelEmbedding( \alpha ) \cdot \alpha = 0$.
  \item An algorithm that computes for given $A,B,T \in \Obj_{\AC}$, $\alpha \in \Hom_{\CC}(A,B)$, $\tau \in \Hom_{\AC}(T,A)$
        such that $\tau \cdot \alpha = 0$
        a morphism $u \in \Hom_{\AC}(T,\kernel( \alpha ))$ 
        such that 
        \[ u \cdot \KernelEmbedding( \alpha ) = \tau,\]
        where $u$ is uniquely determined (up to $=$) by this property
        (\textbf{kernel lift}).
  \item Algorithms that compute for given $A,B \in \Obj_{\AC}$, $\alpha \in \Hom_{\AC}(A,B)$ an object $\cokernel( \alpha )$ (\textbf{cokernel object}) 
        and a morphism
        \[
         \CokernelProjection(\alpha) \in \Hom_{\AC}( B, \cokernel( \alpha )) \text{\hspace{1em}(\textbf{cokernel projection})}
        \]
        such that $\alpha \cdot \CokernelProjection(\alpha) = 0$.
  \item An algorithm that computes for given $A,B,T \in \Obj_{\AC}$ $\alpha \in \Hom_{\AC}(A,B)$, $\tau \in \Hom_{\AC}(B,T)$
        such that $\alpha \cdot \tau = 0$ 
        a morphism $u \in \Hom_{\AC}(\cokernel( \alpha ), T)$
        such that 
        \[
        \CokernelProjection( \alpha ) \cdot u = \tau,
        \]
        where $u$ is uniquely determined (up to $=$) by this property
        (\textbf{cokernel colift}).
  \end{enumerate}
\end{definition}

\begin{remark}
Note that for morphisms $\alpha$, $\alpha'$ such that $\alpha = \alpha'$,
we do not require some ``naive'' equality between the terms
$\kernel( \alpha )$ and $\kernel( \alpha' )$,
but the kernel lift gives us a specific isomorphism.
\end{remark}

\begin{remark}
 Being a \textbf{monomorphism} for $\alpha \in \Hom_{\AC}(A,B)$ in a preabelian category $\AC$
 can be characterized by $\KernelEmbedding( \alpha ) = 0$.
 Dually, being an \textbf{epimorphism} can be characterized by $\CokernelProjection( \alpha ) = 0$.
\end{remark}

\begin{definition}
  An \textbf{abelian category} is a preabelian category $\AC$ for which we have:
  \begin{enumerate}
  \setcounter{enumi}{15}
   \item An algorithm that computes for a given monomorphism $\alpha \in  \Hom_{\AC}(A,B)$ and 
        given morphism $\tau \in \Hom_{\AC}(T,B)$ such that $\tau \cdot \CokernelProjection( \alpha ) = 0$
        the \textbf{lift along a monomorphism} $u \in \Hom_{\AC}( T,A )$ (i.e., $u \cdot \alpha = \tau$).
  \item An algorithm that computes for a given epimorphism $\alpha \in \Hom_{\AC}(A,B)$ and 
        given morphism $\tau \in \Hom_{\AC}(A,T)$ such that $\KernelEmbedding( \alpha ) \cdot \tau = 0$
        the \textbf{colift along an epimorphism} $u \in \Hom_{\AC}(B,T)$ (i.e., $\alpha \cdot u = \tau$).
 \end{enumerate}
\end{definition}

Since we encounter categories with decidable equality for morphisms
as well as with computationally undecidable equality for morphisms (see Remark \ref{remark:abelian_not_computable}),
we introduce a special definition.

\begin{definition}
 A category $\AC$ is called \textbf{computable}
 if we have an algorithm that decides for given $A,B \in \Obj_{\AC}$, $\alpha, \beta \in \Hom_{\AC}(A,B)$
 whether $\alpha = \beta$.
\end{definition}

Note that in \cite[Appendix B]{BL_GabrielMorphisms} decidable equality for morphisms is always assumed.
\end{appendix}

%% file: FreydCats.bbl
\def\cprime{$'$} \def\cprime{$'$} \def\cprime{$'$} \def\cprime{$'$}
  \def\cprime{$'$}
\providecommand{\bysame}{\leavevmode\hbox to3em{\hrulefill}\thinspace}
\providecommand{\MR}{\relax\ifhmode\unskip\space\fi MR }
\providecommand{\MRhref}[2]{%
  \href{http://www.ams.org/mathscinet-getitem?mr=#1}{#2}
}
\providecommand{\href}[2]{#2}